\documentclass[10pt, a4paper, fleqn]{article}
\usepackage{titling}
\title{Bound-Preserving Finite-Volume Schemes for Systems of Continuity Equations with Saturation}

\newcommand{\authorPDF}{Bailo, Carrillo, Hu.}
\newcommand{\subjectPDF}{65M08; 35Q70; 35Q92; 45K05.}
\newcommand{\keywordsPDF}{Gradient flows; systems of continuity equations; finite-volume methods; bound preservation; energy dissipation; integro-differential equations.} \usepackage{authblk}

\author[1,2]{Rafael Bailo}
\author[2]{Jos\'{e} A. Carrillo}
\author[3]{Jingwei Hu}

\affil[1]{
Univ. Lille, CNRS, Inria, UMR 8524 - Laboratoire Paul Painlev\'{e}
}
\affil[ ]{
	F-59000 Lille, France
}
\affil[ ]{\textit{
		rafael.bailo@univ-lille.fr
	}}

\affil[ ]{}

\affil[2]{
	Mathematical Institute, University of Oxford
}
\affil[ ]{
	OX2 6GG Oxford, United Kingdom
}
\affil[ ]{\textit{bailo@maths.ox.ac.uk, carrillo@maths.ox.ac.uk}}

\affil[ ]{}

\affil[3]{
	Department of Applied Mathematics, University of Washington
}
\affil[ ]{
	Seattle, WA 98195, United States of America
}
\affil[ ]{\textit{hujw@uw.edu}} 
\usepackage[left=1.1in, right=1.1in, top=1.1in, bottom=1.1in]{geometry}

\makeatletter
\let\newtitle\@title
\let\newauthor\@author
\let\newdate\@date
\makeatother
\usepackage{fancyhdr}
\usepackage[
	hypertexnames=false,
	pdftex,
	pdftitle={\thetitle. \authorPDF},
	pdfauthor={\authorPDF},
	pdfsubject={\subjectPDF},
	pdfkeywords={\keywordsPDF},
]{hyperref}
\hypersetup{
	colorlinks=true,
	linkcolor=blue,
	citecolor=blue,
	urlcolor=blue,
	linktocpage
}

\usepackage[utf8]{inputenc}
\usepackage[T1]{fontenc}
\usepackage{amsmath}
\usepackage{amsthm}
\usepackage{amsfonts}
\usepackage{amssymb}
\usepackage{graphicx}
\usepackage{mathtools}
\usepackage{ifthen}
\usepackage{dsfont}

\usepackage[UKenglish]{babel}

\usepackage[dvipsnames]{xcolor}

\definecolor{color1}{RGB}{0, 121, 178}
\definecolor{color2}{RGB}{255, 124, 37}
\definecolor{color3}{RGB}{37, 160, 55}
\definecolor{color4}{RGB}{220, 32, 44}
\definecolor{color5}{RGB}{147, 104, 186}
\definecolor{color6}{RGB}{143, 85, 76}
\definecolor{color7}{RGB}{230, 119, 192}
\definecolor{color8}{RGB}{127, 127, 127}
\definecolor{color9}{RGB}{192, 188, 55}
\definecolor{color10}{RGB}{0, 191, 206} 

\newcounter{review}

\newcommand{\ntcreview}[3]{\refstepcounter{review}

	{\color{#2}{\textbf{[#1]}: #3}}}

\newcommand{\creview}[3]{\ntcreview{#1}{#2}{#3}
	\addcontentsline{tor}{subsection}{\thereview~\textbf{[#1]}:~#3
	}}

\newcommand{\review}[2]{\creview{#1}{blue}{#2}}

\makeatletter

\newcommand\listreviewname{List of Reviews}
\newcommand\listofreviews{\section*{\listreviewname}\@starttoc{tor}}
\makeatother

\usepackage[all]{nowidow}

\newcommand{\subjectclassification}[1]{

	{\small\textbf{\textit{AMS Subject Classification --- }} #1}

}
\newcommand{\keywords}[1]{

	{\small\textbf{\textit{Keywords --- }} #1}

}

\renewcommand\lll\MoveEqLeft

\usepackage{tikz}
\usetikzlibrary{patterns}
\usetikzlibrary{decorations.pathreplacing}
\usetikzlibrary{calc}
\usetikzlibrary{fadings}
\usetikzlibrary{shapes}
\usetikzlibrary{plotmarks}
\usetikzlibrary{arrows, decorations.markings}
\tikzset{thicker line small arrows m/.style args={#1in#2}{
			draw=#2,
			solid,
			line width=#1,
			shorten >=1mm,
			decoration={
					markings,
					mark=at position 1.0 with {\arrow[fill=#2,thin]{triangle 90}}
				},
			postaction={decorate}
		}}

\usepackage{pgfplots}
\pgfplotsset{compat=1.16}

\usepackage{float}
\usepackage{subcaption}
\usepackage[section]{placeins}
\usepackage{rotating}

\usepackage{array}
\newcolumntype{L}[1]{>{\raggedright\let\newline\\\arraybackslash\hspace{0pt}}m{#1}}
\newcolumntype{C}[1]{>{\centering\let\newline\\\arraybackslash\hspace{0pt}}m{#1}}
\newcolumntype{R}[1]{>{\raggedleft\let\newline\\\arraybackslash\hspace{0pt}}m{#1}}
\usepackage{booktabs}
\usepackage{tabularx}
\usepackage{multirow}

\usepackage{titling}

\usepackage[algoruled]{algorithm2e}

\usepackage[capitalise]{cleveref}

 \usepackage{accents}

\newcommand\term\emph

\numberwithin{equation}{section}

\usepackage{autobreak}

\usepackage{enumerate}

\makeatletter
\def\@maketitle{\newpage
	\begin{center}\let \footnote \thanks
		{\LARGE\bfseries \@title \par}\vskip 2.5em{\large
				\lineskip .5em\begin{tabular}[t]{c}\@author
				\end{tabular}\par}\vskip 1em{\large \@date}\end{center}\par
	\vskip 1.5em}
\makeatother



\usepackage{amsthm}
\theoremstyle{plain}
\newtheorem{theorem}{Theorem}[section]

\newtheorem{proposition}[theorem]{Proposition}

\theoremstyle{remark}
\newtheorem{remark}[theorem]{\bf Remark}
\newtheorem{definition}[theorem]{\bf Definition}

\usepackage{esint}

\def\XXint#1#2#3{{\setbox0=\hbox{$#1{#2#3}{\int}$ }
			\vcenter{\hbox{$#2#3$ }}\kern-.6\wd0}}

\renewcommand{\th}{\textsuperscript{th} }

\DeclarePairedDelimiter{\prt}{(}{)}
\DeclarePairedDelimiter{\brk}{[}{]}

\DeclarePairedDelimiter{\abs}{|}{|}
\DeclarePairedDelimiter{\norm}{\|}{\|}
\DeclarePairedDelimiter{\set}{\{}{\}}

\DeclarePairedDelimiter{\inn}{\langle}{\rangle}

\usepackage{suffix}

\newcommand{\inner}[2]{\inn{#1,#2}}
\WithSuffix\newcommand\inner*[2]{\inn*{#1,#2}}

\DeclarePairedDelimiter{\positive}{(}{)^{+}}
\DeclarePairedDelimiter{\negative}{(}{)^{-}}

\newcommand\pos\positive
\renewcommand\neg\negative

\WithSuffix\newcommand\pos*{\positive*}
\WithSuffix\newcommand\neg*{\negative*}

\newcommand{\R}{{\mathbb{R}}}

\newcommand{\Rd}{{\mathbb{R}^d}}
\newcommand{\Rplus}{{\mathbb{R}^+}}

\renewcommand{\L}[1]{{L^{#1}}}
\newcommand{\Lone}{\L{1}}
\newcommand{\Ltwo}{\L{2}}
\newcommand{\Lp}{\L{p}}
\newcommand{\Linf}{\L{\infty}}

\newcommand{\pnorm}[2]{\norm{#2}_{\L{#1}}}
\WithSuffix\newcommand\pnorm*[2]{\norm*{#2}_{\L{#1}}}

\newcommand{\psnorm}[3]{\norm{#3}_{\L{#1}(#2)}}
\WithSuffix\newcommand\psnorm*[3]{\norm*{#3}_{\L{#1}(#2)}}

\newcommand{\pnormp}[2]{\pnorm{#1}{#2}^{#1}}
\WithSuffix\newcommand\pnormp*[2]{\pnorm*{#1}{#2}^{#1}}

\newcommand{\psnormp}[3]{\psnorm{#1}{#2}{#3}^{#1}}
\WithSuffix\newcommand\psnormp*[3]{\psnorm*{#1}{#2}{#3}^{#1}}

\newcommand\svec\vec

\renewcommand{\vec}{\mathbf}
\renewcommand{\svec}{\boldsymbol}

\newcommand{\bx}{\svec{x}}

\newcommand{\bu}{\svec{u}}
\newcommand{\bv}{\svec{v}}

\newcommand{\conv}{\ast}

\renewcommand{\d}{\mathrm{d}}
\newcommand{\dd}{\mathop{}\!\d}

\newcommand{\der}[2]{\frac{\d #1}{\d #2}}

\newcommand{\vder}[2]{\frac{\delta #1}{\delta #2}}

\newcommand{\ds}{\dd s}
\newcommand{\dx}{\dd x}
\newcommand{\dy}{\dd y}
\newcommand{\dbx}{\dd \bx}

\newcommand{\grad}{\nabla}
\renewcommand{\div}{\nabla\cdot}

\newcommand{\laplacian}{\Delta}
\newcommand{\laplace}{\laplacian}

\newcommand{\pt}{\partial_t}
\newcommand{\px}{\partial_x}

\newcommand{\Dt}{\Delta t}
\newcommand{\Dx}{\Delta x}
\newcommand{\Dy}{\Delta y}

\newcommand{\nhalf}{1/2}

\renewcommand{\i}{_{i}}
\newcommand{\ip}{_{i+1}}

\newcommand{\ih}{_{i+\nhalf}}
\newcommand{\imh}{_{i-\nhalf}}

\renewcommand{\j}{_{j}}
\newcommand{\jp}{_{j+1}}
\newcommand{\jm}{_{j-1}}

\newcommand{\jmh}{_{j-\nhalf}}

\renewcommand{\k}{_{k}}
\newcommand{\kp}{_{k+1}}

\newcommand{\kh}{_{k+\nhalf}}

\renewcommand{\l}{_{l}}

\newcommand{\n}{^{n}}
\newcommand{\np}{^{n+1}}

\newcommand{\nss}{^{**}}

\newcommand{\ppr}{(r)}

\newcommand{\Wr}{^{W,\,\ppr}}

\newcommand{\Wik}{W_{i-k}}

\newlength{\dhatheight}

\ifthenelse{\isundefined{\Wr}}{
	\newcommand{\Wr}{^{W,\,\ppr}}
}{
	\renewcommand{\Wr}{^{W,\,\ppr}}
}

\newcommand{\s}{\sigma}

\let\goes\rightarrow

\newcommand{\brho}{\svec{\rho}}

\newcommand{\bF}{\svec{F}}

\newcommand{\bone}{\svec{1}}
\DeclareMathOperator{\diag}{diag}

\DeclareMathOperator{\He}{He}

\newcommand{\curlyH}{\mathcal{H}}
\newcommand{\curlyV}{\mathcal{V}}
\newcommand{\curlyW}{\mathcal{W}}

\newcommand{\bxi}{\svec{\xi}}
\newcommand{\bV}{\svec{V}}
\newcommand{\bg}{\svec{g}}

\DeclareMathOperator{\erf}{erf}

\newcommand{\dr}{\dd r}
\newcommand{\dtheta}{\dd \theta}

\newcommand{\bs}{\svec{s}} \usepackage{todonotes}

\DeclareCaptionFormat{cont}{#1 (cont.)#2#3\par}

\newif\ifskiptable

\newcommand{\placedfigure}[1]{
	\begin{figure}\centering
		#1
	\end{figure}
}

\pgfplotsset{colormap={hsv}{
			hsb(0.00cm)=(0.00,0,0.95);
			hsb(0.05cm)=(0.05,1,1);
			hsb(0.10cm)=(0.10,1,1);
			hsb(0.15cm)=(0.15,1,1);
			hsb(0.20cm)=(0.20,1,1);
			hsb(0.25cm)=(0.25,1,1);
			hsb(0.30cm)=(0.30,1,1);
			hsb(0.35cm)=(0.35,1,1);
			hsb(0.40cm)=(0.40,1,1);
			hsb(0.45cm)=(0.45,1,1);
			hsb(0.50cm)=(0.50,1,1);
			hsb(0.55cm)=(0.55,1,1);
			hsb(0.60cm)=(0.60,1,1);
			hsb(0.65cm)=(0.65,1,1);
			hsb(0.70cm)=(0.70,1,1);
			hsb(0.75cm)=(0.75,1,1);
			hsb(0.80cm)=(0.80,1,1);
			hsb(0.85cm)=(0.85,1,1);
			hsb(0.90cm)=(0.90,1,1);
			hsb(0.95cm)=(0.95,1,1);
			hsb(1.00cm)=(1.00,1,1);
		}
}

\pgfplotsset{colormap={hsvSoft}{
			hsb(0.00cm)=(0.00,0,0.95);
			hsb(0.05cm)=(0.05,1,1);
			hsb(0.10cm)=(0.10,1,1);
			hsb(0.15cm)=(0.15,1,1);
			hsb(0.20cm)=(0.20,1,1);
			hsb(0.25cm)=(0.25,1,1);
			hsb(0.30cm)=(0.30,1,1);
			hsb(0.35cm)=(0.35,1,1);
			hsb(0.40cm)=(0.40,1,1);
			hsb(0.45cm)=(0.45,1,1);
			hsb(0.50cm)=(0.50,1,1);
			hsb(0.55cm)=(0.55,1,1);
			hsb(0.60cm)=(0.60,1,1);
			hsb(0.65cm)=(0.65,1,1);
			hsb(0.70cm)=(0.70,1,1);
			hsb(0.75cm)=(0.75,1,1);
			hsb(0.80cm)=(0.80,1,1);
			hsb(0.85cm)=(0.85,1,1);
			hsb(0.90cm)=(0.90,1,1);
			hsb(0.95cm)=(0.95,1,1);
			hsb(1.00cm)=(0.00,0,0.95);
		}
}

\pgfplotsset{colormap={viridisFull}{
			rgb=(0.26700401, 0.00487433, 0.32941519)
			rgb=(0.26851048, 0.00960483, 0.33542652)
			rgb=(0.26994384, 0.01462494, 0.34137895)
			rgb=(0.27130489, 0.01994186, 0.34726862)
			rgb=(0.27259384, 0.02556309, 0.35309303)
			rgb=(0.27380934, 0.03149748, 0.35885256)
			rgb=(0.27495242, 0.03775181, 0.36454323)
			rgb=(0.27602238, 0.04416723, 0.37016418)
			rgb=(0.2770184 , 0.05034437, 0.37571452)
			rgb=(0.27794143, 0.05632444, 0.38119074)
			rgb=(0.27879067, 0.06214536, 0.38659204)
			rgb=(0.2795655 , 0.06783587, 0.39191723)
			rgb=(0.28026658, 0.07341724, 0.39716349)
			rgb=(0.28089358, 0.07890703, 0.40232944)
			rgb=(0.28144581, 0.0843197 , 0.40741404)
			rgb=(0.28192358, 0.08966622, 0.41241521)
			rgb=(0.28232739, 0.09495545, 0.41733086)
			rgb=(0.28265633, 0.10019576, 0.42216032)
			rgb=(0.28291049, 0.10539345, 0.42690202)
			rgb=(0.28309095, 0.11055307, 0.43155375)
			rgb=(0.28319704, 0.11567966, 0.43611482)
			rgb=(0.28322882, 0.12077701, 0.44058404)
			rgb=(0.28318684, 0.12584799, 0.44496 )
			rgb=(0.283072 , 0.13089477, 0.44924127)
			rgb=(0.28288389, 0.13592005, 0.45342734)
			rgb=(0.28262297, 0.14092556, 0.45751726)
			rgb=(0.28229037, 0.14591233, 0.46150995)
			rgb=(0.28188676, 0.15088147, 0.46540474)
			rgb=(0.28141228, 0.15583425, 0.46920128)
			rgb=(0.28086773, 0.16077132, 0.47289909)
			rgb=(0.28025468, 0.16569272, 0.47649762)
			rgb=(0.27957399, 0.17059884, 0.47999675)
			rgb=(0.27882618, 0.1754902 , 0.48339654)
			rgb=(0.27801236, 0.18036684, 0.48669702)
			rgb=(0.27713437, 0.18522836, 0.48989831)
			rgb=(0.27619376, 0.19007447, 0.49300074)
			rgb=(0.27519116, 0.1949054 , 0.49600488)
			rgb=(0.27412802, 0.19972086, 0.49891131)
			rgb=(0.27300596, 0.20452049, 0.50172076)
			rgb=(0.27182812, 0.20930306, 0.50443413)
			rgb=(0.27059473, 0.21406899, 0.50705243)
			rgb=(0.26930756, 0.21881782, 0.50957678)
			rgb=(0.26796846, 0.22354911, 0.5120084 )
			rgb=(0.26657984, 0.2282621 , 0.5143487 )
			rgb=(0.2651445 , 0.23295593, 0.5165993 )
			rgb=(0.2636632 , 0.23763078, 0.51876163)
			rgb=(0.26213801, 0.24228619, 0.52083736)
			rgb=(0.26057103, 0.2469217 , 0.52282822)
			rgb=(0.25896451, 0.25153685, 0.52473609)
			rgb=(0.25732244, 0.2561304 , 0.52656332)
			rgb=(0.25564519, 0.26070284, 0.52831152)
			rgb=(0.25393498, 0.26525384, 0.52998273)
			rgb=(0.25219404, 0.26978306, 0.53157905)
			rgb=(0.25042462, 0.27429024, 0.53310261)
			rgb=(0.24862899, 0.27877509, 0.53455561)
			rgb=(0.2468114 , 0.28323662, 0.53594093)
			rgb=(0.24497208, 0.28767547, 0.53726018)
			rgb=(0.24311324, 0.29209154, 0.53851561)
			rgb=(0.24123708, 0.29648471, 0.53970946)
			rgb=(0.23934575, 0.30085494, 0.54084398)
			rgb=(0.23744138, 0.30520222, 0.5419214 )
			rgb=(0.23552606, 0.30952657, 0.54294396)
			rgb=(0.23360277, 0.31382773, 0.54391424)
			rgb=(0.2316735 , 0.3181058 , 0.54483444)
			rgb=(0.22973926, 0.32236127, 0.54570633)
			rgb=(0.22780192, 0.32659432, 0.546532 )
			rgb=(0.2258633 , 0.33080515, 0.54731353)
			rgb=(0.22392515, 0.334994 , 0.54805291)
			rgb=(0.22198915, 0.33916114, 0.54875211)
			rgb=(0.22005691, 0.34330688, 0.54941304)
			rgb=(0.21812995, 0.34743154, 0.55003755)
			rgb=(0.21620971, 0.35153548, 0.55062743)
			rgb=(0.21429757, 0.35561907, 0.5511844 )
			rgb=(0.21239477, 0.35968273, 0.55171011)
			rgb=(0.2105031 , 0.36372671, 0.55220646)
			rgb=(0.20862342, 0.36775151, 0.55267486)
			rgb=(0.20675628, 0.37175775, 0.55311653)
			rgb=(0.20490257, 0.37574589, 0.55353282)
			rgb=(0.20306309, 0.37971644, 0.55392505)
			rgb=(0.20123854, 0.38366989, 0.55429441)
			rgb=(0.1994295 , 0.38760678, 0.55464205)
			rgb=(0.1976365 , 0.39152762, 0.55496905)
			rgb=(0.19585993, 0.39543297, 0.55527637)
			rgb=(0.19410009, 0.39932336, 0.55556494)
			rgb=(0.19235719, 0.40319934, 0.55583559)
			rgb=(0.19063135, 0.40706148, 0.55608907)
			rgb=(0.18892259, 0.41091033, 0.55632606)
			rgb=(0.18723083, 0.41474645, 0.55654717)
			rgb=(0.18555593, 0.4185704 , 0.55675292)
			rgb=(0.18389763, 0.42238275, 0.55694377)
			rgb=(0.18225561, 0.42618405, 0.5571201 )
			rgb=(0.18062949, 0.42997486, 0.55728221)
			rgb=(0.17901879, 0.43375572, 0.55743035)
			rgb=(0.17742298, 0.4375272 , 0.55756466)
			rgb=(0.17584148, 0.44128981, 0.55768526)
			rgb=(0.17427363, 0.4450441 , 0.55779216)
			rgb=(0.17271876, 0.4487906 , 0.55788532)
			rgb=(0.17117615, 0.4525298 , 0.55796464)
			rgb=(0.16964573, 0.45626209, 0.55803034)
			rgb=(0.16812641, 0.45998802, 0.55808199)
			rgb=(0.1666171 , 0.46370813, 0.55811913)
			rgb=(0.16511703, 0.4674229 , 0.55814141)
			rgb=(0.16362543, 0.47113278, 0.55814842)
			rgb=(0.16214155, 0.47483821, 0.55813967)
			rgb=(0.16066467, 0.47853961, 0.55811466)
			rgb=(0.15919413, 0.4822374 , 0.5580728 )
			rgb=(0.15772933, 0.48593197, 0.55801347)
			rgb=(0.15626973, 0.4896237 , 0.557936 )
			rgb=(0.15481488, 0.49331293, 0.55783967)
			rgb=(0.15336445, 0.49700003, 0.55772371)
			rgb=(0.1519182 , 0.50068529, 0.55758733)
			rgb=(0.15047605, 0.50436904, 0.55742968)
			rgb=(0.14903918, 0.50805136, 0.5572505 )
			rgb=(0.14760731, 0.51173263, 0.55704861)
			rgb=(0.14618026, 0.51541316, 0.55682271)
			rgb=(0.14475863, 0.51909319, 0.55657181)
			rgb=(0.14334327, 0.52277292, 0.55629491)
			rgb=(0.14193527, 0.52645254, 0.55599097)
			rgb=(0.14053599, 0.53013219, 0.55565893)
			rgb=(0.13914708, 0.53381201, 0.55529773)
			rgb=(0.13777048, 0.53749213, 0.55490625)
			rgb=(0.1364085 , 0.54117264, 0.55448339)
			rgb=(0.13506561, 0.54485335, 0.55402906)
			rgb=(0.13374299, 0.54853458, 0.55354108)
			rgb=(0.13244401, 0.55221637, 0.55301828)
			rgb=(0.13117249, 0.55589872, 0.55245948)
			rgb=(0.1299327 , 0.55958162, 0.55186354)
			rgb=(0.12872938, 0.56326503, 0.55122927)
			rgb=(0.12756771, 0.56694891, 0.55055551)
			rgb=(0.12645338, 0.57063316, 0.5498411 )
			rgb=(0.12539383, 0.57431754, 0.54908564)
			rgb=(0.12439474, 0.57800205, 0.5482874 )
			rgb=(0.12346281, 0.58168661, 0.54744498)
			rgb=(0.12260562, 0.58537105, 0.54655722)
			rgb=(0.12183122, 0.58905521, 0.54562298)
			rgb=(0.12114807, 0.59273889, 0.54464114)
			rgb=(0.12056501, 0.59642187, 0.54361058)
			rgb=(0.12009154, 0.60010387, 0.54253043)
			rgb=(0.11973756, 0.60378459, 0.54139999)
			rgb=(0.11951163, 0.60746388, 0.54021751)
			rgb=(0.11942341, 0.61114146, 0.53898192)
			rgb=(0.11948255, 0.61481702, 0.53769219)
			rgb=(0.11969858, 0.61849025, 0.53634733)
			rgb=(0.12008079, 0.62216081, 0.53494633)
			rgb=(0.12063824, 0.62582833, 0.53348834)
			rgb=(0.12137972, 0.62949242, 0.53197275)
			rgb=(0.12231244, 0.63315277, 0.53039808)
			rgb=(0.12344358, 0.63680899, 0.52876343)
			rgb=(0.12477953, 0.64046069, 0.52706792)
			rgb=(0.12632581, 0.64410744, 0.52531069)
			rgb=(0.12808703, 0.64774881, 0.52349092)
			rgb=(0.13006688, 0.65138436, 0.52160791)
			rgb=(0.13226797, 0.65501363, 0.51966086)
			rgb=(0.13469183, 0.65863619, 0.5176488 )
			rgb=(0.13733921, 0.66225157, 0.51557101)
			rgb=(0.14020991, 0.66585927, 0.5134268 )
			rgb=(0.14330291, 0.66945881, 0.51121549)
			rgb=(0.1466164 , 0.67304968, 0.50893644)
			rgb=(0.15014782, 0.67663139, 0.5065889 )
			rgb=(0.15389405, 0.68020343, 0.50417217)
			rgb=(0.15785146, 0.68376525, 0.50168574)
			rgb=(0.16201598, 0.68731632, 0.49912906)
			rgb=(0.1663832 , 0.69085611, 0.49650163)
			rgb=(0.1709484 , 0.69438405, 0.49380294)
			rgb=(0.17570671, 0.6978996 , 0.49103252)
			rgb=(0.18065314, 0.70140222, 0.48818938)
			rgb=(0.18578266, 0.70489133, 0.48527326)
			rgb=(0.19109018, 0.70836635, 0.48228395)
			rgb=(0.19657063, 0.71182668, 0.47922108)
			rgb=(0.20221902, 0.71527175, 0.47608431)
			rgb=(0.20803045, 0.71870095, 0.4728733 )
			rgb=(0.21400015, 0.72211371, 0.46958774)
			rgb=(0.22012381, 0.72550945, 0.46622638)
			rgb=(0.2263969 , 0.72888753, 0.46278934)
			rgb=(0.23281498, 0.73224735, 0.45927675)
			rgb=(0.2393739 , 0.73558828, 0.45568838)
			rgb=(0.24606968, 0.73890972, 0.45202405)
			rgb=(0.25289851, 0.74221104, 0.44828355)
			rgb=(0.25985676, 0.74549162, 0.44446673)
			rgb=(0.26694127, 0.74875084, 0.44057284)
			rgb=(0.27414922, 0.75198807, 0.4366009 )
			rgb=(0.28147681, 0.75520266, 0.43255207)
			rgb=(0.28892102, 0.75839399, 0.42842626)
			rgb=(0.29647899, 0.76156142, 0.42422341)
			rgb=(0.30414796, 0.76470433, 0.41994346)
			rgb=(0.31192534, 0.76782207, 0.41558638)
			rgb=(0.3198086 , 0.77091403, 0.41115215)
			rgb=(0.3277958 , 0.77397953, 0.40664011)
			rgb=(0.33588539, 0.7770179 , 0.40204917)
			rgb=(0.34407411, 0.78002855, 0.39738103)
			rgb=(0.35235985, 0.78301086, 0.39263579)
			rgb=(0.36074053, 0.78596419, 0.38781353)
			rgb=(0.3692142 , 0.78888793, 0.38291438)
			rgb=(0.37777892, 0.79178146, 0.3779385 )
			rgb=(0.38643282, 0.79464415, 0.37288606)
			rgb=(0.39517408, 0.79747541, 0.36775726)
			rgb=(0.40400101, 0.80027461, 0.36255223)
			rgb=(0.4129135 , 0.80304099, 0.35726893)
			rgb=(0.42190813, 0.80577412, 0.35191009)
			rgb=(0.43098317, 0.80847343, 0.34647607)
			rgb=(0.44013691, 0.81113836, 0.3409673 )
			rgb=(0.44936763, 0.81376835, 0.33538426)
			rgb=(0.45867362, 0.81636288, 0.32972749)
			rgb=(0.46805314, 0.81892143, 0.32399761)
			rgb=(0.47750446, 0.82144351, 0.31819529)
			rgb=(0.4870258 , 0.82392862, 0.31232133)
			rgb=(0.49661536, 0.82637633, 0.30637661)
			rgb=(0.5062713 , 0.82878621, 0.30036211)
			rgb=(0.51599182, 0.83115784, 0.29427888)
			rgb=(0.52577622, 0.83349064, 0.2881265 )
			rgb=(0.5356211 , 0.83578452, 0.28190832)
			rgb=(0.5455244 , 0.83803918, 0.27562602)
			rgb=(0.55548397, 0.84025437, 0.26928147)
			rgb=(0.5654976 , 0.8424299 , 0.26287683)
			rgb=(0.57556297, 0.84456561, 0.25641457)
			rgb=(0.58567772, 0.84666139, 0.24989748)
			rgb=(0.59583934, 0.84871722, 0.24332878)
			rgb=(0.60604528, 0.8507331 , 0.23671214)
			rgb=(0.61629283, 0.85270912, 0.23005179)
			rgb=(0.62657923, 0.85464543, 0.22335258)
			rgb=(0.63690157, 0.85654226, 0.21662012)
			rgb=(0.64725685, 0.85839991, 0.20986086)
			rgb=(0.65764197, 0.86021878, 0.20308229)
			rgb=(0.66805369, 0.86199932, 0.19629307)
			rgb=(0.67848868, 0.86374211, 0.18950326)
			rgb=(0.68894351, 0.86544779, 0.18272455)
			rgb=(0.69941463, 0.86711711, 0.17597055)
			rgb=(0.70989842, 0.86875092, 0.16925712)
			rgb=(0.72039115, 0.87035015, 0.16260273)
			rgb=(0.73088902, 0.87191584, 0.15602894)
			rgb=(0.74138803, 0.87344918, 0.14956101)
			rgb=(0.75188414, 0.87495143, 0.14322828)
			rgb=(0.76237342, 0.87642392, 0.13706449)
			rgb=(0.77285183, 0.87786808, 0.13110864)
			rgb=(0.78331535, 0.87928545, 0.12540538)
			rgb=(0.79375994, 0.88067763, 0.12000532)
			rgb=(0.80418159, 0.88204632, 0.11496505)
			rgb=(0.81457634, 0.88339329, 0.11034678)
			rgb=(0.82494028, 0.88472036, 0.10621724)
			rgb=(0.83526959, 0.88602943, 0.1026459 )
			rgb=(0.84556056, 0.88732243, 0.09970219)
			rgb=(0.8558096 , 0.88860134, 0.09745186)
			rgb=(0.86601325, 0.88986815, 0.09595277)
			rgb=(0.87616824, 0.89112487, 0.09525046)
			rgb=(0.88627146, 0.89237353, 0.09537439)
			rgb=(0.89632002, 0.89361614, 0.09633538)
			rgb=(0.90631121, 0.89485467, 0.09812496)
			rgb=(0.91624212, 0.89609127, 0.1007168 )
			rgb=(0.92610579, 0.89732977, 0.10407067)
			rgb=(0.93590444, 0.8985704 , 0.10813094)
			rgb=(0.94563626, 0.899815 , 0.11283773)
			rgb=(0.95529972, 0.90106534, 0.11812832)
			rgb=(0.96489353, 0.90232311, 0.12394051)
			rgb=(0.97441665, 0.90358991, 0.13021494)
			rgb=(0.98386829, 0.90486726, 0.13689671)
			rgb=(0.99324789, 0.90615657, 0.1439362 )
		}
}

\pgfplotsset{colormap={viridisSoft}{
			rgb255=(242, 242, 242);
rgb=(0.28026,0.1657,0.4765);
			rgb=(0.26366,0.23763,0.51877);
			rgb=(0.23744,0.3052,0.54192);
			rgb=(0.20862,0.36775,0.55267);
			rgb=(0.18225,0.42618,0.55711);
			rgb=(0.1592,0.48224,0.55807);
			rgb=(0.13777,0.53749,0.5549);
			rgb=(0.12115,0.59274,0.54465);
			rgb=(0.12808,0.64775,0.5235);
			rgb=(0.18065,0.7014,0.48819);
			rgb=(0.27415,0.75198,0.4366);
			rgb=(0.39517,0.79747,0.36775);
			rgb=(0.53561,0.83578,0.2819);
			rgb=(0.68895,0.86545,0.18272);
			rgb=(0.84557,0.88733,0.0997);
			rgb=(0.99324,0.90616,0.14394)
		}
}

\pgfplotsset{colormap={cellRed}{
			rgb255=(242.0,242.0,242.0);
			rgb255=(241.63157894736844,234.47368421052633,234.47368421052633);
			rgb255=(241.26315789473685,226.94736842105266,226.94736842105266);
			rgb255=(240.89473684210526,219.42105263157893,219.42105263157893);
			rgb255=(240.5263157894737,211.89473684210526,211.89473684210526);
			rgb255=(240.1578947368421,204.3684210526316,204.3684210526316);
			rgb255=(239.78947368421052,196.84210526315792,196.84210526315792);
			rgb255=(239.42105263157896,189.31578947368422,189.31578947368422);
			rgb255=(239.05263157894737,181.78947368421052,181.78947368421052);
			rgb255=(238.6842105263158,174.26315789473688,174.26315789473688);
			rgb255=(238.31578947368422,166.73684210526315,166.73684210526315);
			rgb255=(237.94736842105263,159.21052631578948,159.21052631578948);
			rgb255=(237.57894736842104,151.68421052631578,151.68421052631578);
			rgb255=(237.21052631578948,144.1578947368421,144.1578947368421);
			rgb255=(236.84210526315792,136.63157894736844,136.63157894736844);
			rgb255=(236.47368421052633,129.10526315789474,129.10526315789474);
			rgb255=(236.10526315789474,121.57894736842107,121.57894736842107);
			rgb255=(235.73684210526318,114.05263157894737,114.05263157894737);
			rgb255=(235.3684210526316,106.52631578947368,106.52631578947368);
			rgb255=(235.0,99.0,99.0);
		}
}

\pgfplotsset{colormap={cellGreen}{
			rgb255=(242.0,242.0,242.0);
			rgb255=(236.21052631578948,239.5263157894737,234.26315789473685);
			rgb255=(230.42105263157896,237.05263157894737,226.5263157894737);
			rgb255=(224.6315789473684,234.57894736842104,218.78947368421052);
			rgb255=(218.8421052631579,232.10526315789474,211.05263157894737);
			rgb255=(213.05263157894737,229.63157894736844,203.31578947368422);
			rgb255=(207.26315789473685,227.1578947368421,195.57894736842107);
			rgb255=(201.4736842105263,224.68421052631578,187.8421052631579);
			rgb255=(195.68421052631578,222.21052631578948,180.10526315789474);
			rgb255=(189.8947368421053,219.73684210526318,172.36842105263162);
			rgb255=(184.10526315789474,217.26315789473682,164.63157894736844);
			rgb255=(178.31578947368422,214.78947368421052,156.89473684210526);
			rgb255=(172.5263157894737,212.31578947368422,149.1578947368421);
			rgb255=(166.73684210526318,209.84210526315792,141.42105263157896);
			rgb255=(160.94736842105263,207.3684210526316,133.6842105263158);
			rgb255=(155.1578947368421,204.89473684210526,125.94736842105263);
			rgb255=(149.3684210526316,202.42105263157893,118.21052631578948);
			rgb255=(143.57894736842104,199.94736842105266,110.47368421052632);
			rgb255=(137.78947368421052,197.47368421052633,102.73684210526316);
			rgb255=(132.0,195.0,95.0);
		}
}

\pgfplotsset{colormap={cellRedSquared}{
			rgb255=(242.0,242.0,242.0);
			rgb255=(241.28254847645428,227.34349030470915,227.34349030470915);
			rgb255=(240.60387811634348,213.47922437673128,213.47922437673128);
			rgb255=(239.9639889196676,200.40720221606648,200.40720221606648);
			rgb255=(239.36288088642658,188.1274238227147,188.1274238227147);
			rgb255=(238.8005540166205,176.63988919667594,176.63988919667594);
			rgb255=(238.2770083102493,165.94459833795014,165.94459833795014);
			rgb255=(237.79224376731304,156.04155124653738,156.04155124653738);
			rgb255=(237.34626038781164,146.93074792243766,146.93074792243766);
			rgb255=(236.93905817174516,138.61218836565098,138.61218836565098);
			rgb255=(236.57063711911357,131.0858725761773,131.0858725761773);
			rgb255=(236.2409972299169,124.35180055401662,124.35180055401662);
			rgb255=(235.95013850415512,118.40997229916897,118.40997229916897);
			rgb255=(235.69806094182823,113.26038781163435,113.26038781163435);
			rgb255=(235.4847645429363,108.90304709141274,108.90304709141274);
			rgb255=(235.3102493074792,105.33795013850416,105.33795013850416);
			rgb255=(235.17451523545705,102.56509695290858,102.56509695290858);
			rgb255=(235.0775623268698,100.58448753462605,100.58448753462605);
			rgb255=(235.01939058171746,99.3961218836565,99.3961218836565);
			rgb255=(235.0,99.0,99.0);
		}
}
\pgfplotsset{colormap={cellGreenSquared}{
			rgb255=(242.0,242.0,242.0);
			rgb255=(230.7257617728532,237.18282548476455,226.93351800554018);
			rgb255=(220.06094182825484,232.62603878116343,212.6814404432133);
			rgb255=(210.00554016620498,228.32963988919667,199.2437673130194);
			rgb255=(200.5595567867036,224.29362880886427,186.62049861495845);
			rgb255=(191.7229916897507,220.5180055401662,174.8116343490305);
			rgb255=(183.49584487534625,217.0027700831025,163.81717451523545);
			rgb255=(175.87811634349032,213.74792243767314,153.63711911357342);
			rgb255=(168.86980609418282,210.75346260387812,144.27146814404432);
			rgb255=(162.47091412742384,208.01939058171746,135.72022160664818);
			rgb255=(156.68144044321332,205.54570637119116,127.98337950138506);
			rgb255=(151.50138504155126,203.33240997229916,121.06094182825484);
			rgb255=(146.9307479224377,201.37950138504155,114.95290858725764);
			rgb255=(142.96952908587255,199.68698060941827,109.65927977839334);
			rgb255=(139.61772853185596,198.25484764542935,105.18005540166205);
			rgb255=(136.8753462603878,197.0831024930748,101.51523545706371);
			rgb255=(134.74238227146813,196.17174515235456,98.66481994459834);
			rgb255=(133.21883656509695,195.5207756232687,96.62880886426592);
			rgb255=(132.30470914127426,195.13019390581718,95.40720221606648);
			rgb255=(132.0,195.0,95.0);
		}
}

\usepgfplotslibrary{patchplots}

\pgfplotsset{every axis/.append style={
			grid=both,
			grid style={white, line width=.1pt},
			major grid style={white, line width=1.5pt},
			axis background/.style={fill=gray!10},
			axis line style={draw=none},
			tick style={draw=none},
			xlabel = $x$,
line width=1pt,
legend style={
					line width = 1pt,
					draw=none,
					/tikz/every even column/.append style={column sep=0.5cm}
				},
		}}

\usepgfplotslibrary{groupplots}

\usepackage{calc}

\definecolor{gg0}{HTML}{E24A33}
\definecolor{gg1}{HTML}{348ABD}
\definecolor{gg2}{HTML}{988ED5}
\definecolor{gg3}{HTML}{777777}
\definecolor{gg4}{HTML}{FBC15E}
\definecolor{gg5}{HTML}{8EBA42}
\definecolor{gg6}{HTML}{FFB5B8}

\pgfplotsset{
	/pgfplots/colormap={bright}{rgb255=(0,0,0) rgb255=(78,3,100) rgb255=(2,74,255)
			rgb255=(255,21,181) rgb255=(255,113,26) rgb255=(147,213,114) rgb255=(230,255,0)
			rgb255=(255,255,255)}
}

\renewcommand{\review}[2]{}
\renewcommand{\creview}[3]{}
\renewcommand{\ntcreview}[3]{}

\renewcommand{\tableofcontents}{}
\renewcommand{\listofreviews}{}

\expandafter\def\csname ver@etex.sty\endcsname{3000/12/31}

\usepackage{autonum}
\makeatletter
\patchcmd{\autonum@saveEnvironmentSubcommands}
{(0,0)\begin}
{(0,0)\hfuzz=\maxdimen\begin}
{}{}
\makeatother

\AtBeginDocument{

}

\makeatletter
\autonum@generatePatchedReferenceCSL{ref}
\autonum@generatePatchedReferenceCSL{eqref}
\autonum@generatePatchedReferenceCSL{Cref}
\makeatother

\newcommand{\tocnewpage}{}
 
\newcommand{\revision}[2]{#2}
\newcommand{\revisionNote}[2]{}

\definecolor{revisionColourOne}{RGB}{180,0,0}
\definecolor{revisionColourTwo}{RGB}{0,0,180}

\newcommand{\revisionOne}[1]{\revision{revisionColourOne}{#1}}
\newcommand{\revisionTwo}[1]{\revision{revisionColourTwo}{#1}}

\newcommand{\revisionNoteOne}[1]{\revisionNote{revisionColourOne}{#1}}
\newcommand{\revisionNoteTwo}[1]{\revisionNote{revisionColourTwo}{#1}}

\usepackage{setspace}
\begin{document}
\begin{singlespace}\maketitle\end{singlespace}
\begin{abstract}
	We propose finite-volume schemes for general continuity equations which preserve positivity and global bounds that arise from saturation effects in the mobility function. In the case of gradient flows, the schemes dissipate the free energy at the fully discrete level. Moreover, these schemes are generalised to coupled systems of non-linear continuity equations, such as multispecies models in mathematical physics or biology, preserving the bounds and the dissipation of the energy whenever applicable. These results are illustrated through extensive numerical simulations which explore known behaviours in biology and showcase new phenomena not yet described by the literature.
\end{abstract}
\subjectclassification{\subjectPDF}
\tocnewpage
\tableofcontents
\tocnewpage
\listofreviews
\tocnewpage

\revisionNoteOne{Changes corresponding to the comments of Reviewer 1 are shown in red.}
\revisionNoteTwo{Similarly, changes corresponding to Reviewer 2 are shown in blue.}

\section{Introduction}

Systems of aggregation-diffusion equations are ubiquitous in science and engineering, particularly in mathematical biology, where population dynamics models have a distinctive importance. These systems faithfully reproduce attraction effects, usually modelled by non-local terms; repulsion effects, typically encoded by non-linear diffusion and cross-diffusion; and global constraints, usually due to volume exclusion effects within the populations. The balance and interplay between these terms, which leads interesting phenomena in single population models \cite{CHS18,CCY19}, can produce even richer behaviours in systems of multiple species \cite{MURAKAWA20151,C.M.S+2019}. For instance, cell-sorting phenomena, triggered by differential adhesion, has been reported in aggregation-reaction-diffusion systems \cite{BDFS18,Trush5861}; this biological behaviour is sharply captured by these models, matching experimental data well \cite{C.M.S+2019}.

Many of these models take the form of a system of gradient flows:
\begin{equation}\label{eq:multispeciesgradientflowintro}
	\pt \brho
	= \div\prt*{M\prt{\brho} \psi(\sigma) \grad\prt*{\vder{E}{\brho}}}
	= \sum_{l=1}^{d} \partial_{x\l}\prt*{M\prt{\brho} \psi(\sigma) \partial_{x\l}\prt*{\vder{E}{\brho}}}
\end{equation}
for a given energy functional $E$. Here, $\brho$ is a vector $\brho=\prt{\rho_1, \rho_2, \cdots, \rho_P}^\top$ of the densities of the $P$ species, and $\rho_p\prt{t,\bx}: \Rplus\times\Omega \rightarrow \Rplus$. $M\prt{\brho}$ is a $P\times P$ positive semi-definite matrix, possibly non-symmetric.
Additionally, $\psi(\s):\R^+\rightarrow\R$ is a \textit{saturation}: a decreasing function of the \textit{total density} of the system, $\s\prt{t,\bx}\coloneqq\sum_{p=1}^{P} \rho_p\prt{t,\bx}$, such that $\psi(\alpha)=0$ at a given \textit{saturation level} $\alpha$, viz. \cref{def:saturation}. The energy functional is key to the analysis of this problem, since it provides a formal dissipation estimate:
\begin{equation}
	\der{}{t}E\brk{\brho}
	= \int_\Omega \vder{E}{\brho} \cdot \pt \brho
	=
	\revisionOne{
		-\sum_{l=1}^{d}
		\int_\Omega \psi\prt{\sigma} \partial_{x\l}\prt*{\vder{E}{\brho}} \cdot M\prt{\brho} \partial_{x\l}\prt*{\vder{E}{\brho}} \dbx
	}
	\leq 0,
\end{equation}
since $M\prt{\brho}$ is positive semi-definite. The form of the energy functional can lead to a large range of phenomena, including linear/non-linear diffusion and local/non-local drifts \cite{C.M.V2003}.

The analysis of cross-diffusion systems which include potential or non-local terms is nevertheless quite challenging. We highlight a series of recent works \cite{jungel2015boundedness, jungelzamponi,jungelzamponi2} which deal with a family of cross-diffusion systems without drift terms; they exploit the gradient-flow structure of the systems to define global solutions with global pointwise bounds. The entropy dissipation techniques used there are not only essential to define the solutions, but also to determine their long-time asymptotics. Similar results have only been achieved for systems with drift terms in a handful of specific cases \cite{BDFS18,BCPS20}.

This work is concerned with the design of numerical schemes that preserve the structural properties of the solutions to \eqref{eq:multispeciesgradientflowintro}, and thus, that are suitable to explore their rich qualitative behaviours. To that end, two crucial aspects beg consideration. The first one is the saturation effect, due to $\psi(\sigma)$, which arises in these models due to volume exclusion effects, see for instance \cite{DBLP:journals/nhm/AlmeidaBPP19}; the total density of the system has a natural global bound, which must be preserved by the numerical scheme, along with the non-negativity of the density of each species. The second one is the gradient-flow structure, present in many applications; in these cases, a scheme which captures the dissipation at the fully discrete level (a \textit{structure-preserving scheme}) is desirable. We emphasise that the saturation terms affect the rate of energy dissipation of the system, and therefore their proper discretisation is required for this goal.

While numerical works for continuity equations are abundant (\cite{CCH15,PZ18,S.C.S2018,B.C.H2020,E.R2020} among others), few deal with the saturation effects as well as systems \cite{S.C.S2019,LZ21}. The work \cite{DBLP:journals/nhm/AlmeidaBPP19} proposes an upwind scheme for scalar Fokker-Planck-like equations in chemotaxis, where the saturation effects only act on the drift term. A series of works by J\"{u}ngel and collaborators (\cite{jungel2021discrete} and the references therein) propose finite-volume schemes for cross-diffusion systems without drifts; there, a hard volume-filling constraint is imposed by expressing one solution variable in terms of the rest.

In this work we will propose finite-volume schemes for the general formulation \eqref{eq:multispeciesgradientflowintro} which preserve both the positivity and the global bounds induced by the saturation effects, and they also guarantee the dissipation of the free energy at the fully discrete level. Based on our previous work \cite{B.C.H2020}, we consider implicit schemes where the numerical fluxes are carefully chosen to preserve the lower and upper bounds, a property which later plays a crucial role in preserving the gradient-flow structure. Our numerical methods combine techniques from systems of hyperbolic conservation laws \cite{H.R.T1995,H.R2015}, such as upwinding, with a careful discretization of the velocity fields \cite{CCH15,B.C.H2020} in order to capture the desired dissipative structure.

We will demonstrate the prowess of our methods through a series of numerical examples. We remark that no general theoretical results for equations or systems with saturation are available in the literature, except for a few cases \cite{CLSS10,H.R.T1995,H.R2015}, making our numerical explorations very interesting in terms of new phenomena.

Our work is organised in a constructive way. We begin by discussing the case of scalar gradient flows with saturation in \cref{sec:newGradientFlowsScalar}, where we design an implicit numerical scheme, and prove that it preserves the bounds and dissipation structure of the equation unconditionally. The scheme is generalised in \cref{sec:newGradientFlowsSystems} to the case of systems of gradient flows, again demonstrating the unconditional structure-preserving properties at the fully discrete level. To conclude, \cref{sec:numericalExperiments} is devoted to numerical experiments which validate the numerical schemes and showcase their use. These experiments illustrate novel phenomena which arises from the saturation effects, not yet described by the available literature.

\section{Scalar Gradient Flows with Saturation}\label{sec:newGradientFlowsScalar}

We begin this work by studying a family of scalar general flow equations of the form
\begin{equation}\label{eq:kineticgradientflowsaturation}
	\begin{cases}
		\pt\rho = \div\brk{\rho \psi\prt{\rho} \grad \prt{H'(\rho) + V + W\conv\rho}}, \\
		\rho(0,\bx) = \rho_0(\bx),
	\end{cases}
\end{equation}
for $t>0$ and $\bx\in\Omega\subseteq \mathbb{R}^d$.
This equation describes the evolution of a \textit{density} of particles $\rho\prt{t, \bx}: \Rplus\times\Omega\rightarrow\Rplus$, an unknown non-negative function with a prescribed initial datum $\rho_0(\bx) : \Omega\rightarrow\Rplus$. The \textit{internal energy density} $H(\rho)$, a convex function, models linear or non-linear diffusion on $\rho$. The \textit{confining potential} $V(\bx)$ models external forces acting on the density. The \textit{interaction potential} $W(\bx)$, a symmetric function (typically radial), models the attraction or repulsion between particles.

\Cref{eq:kineticgradientflowsaturation}, which can be interpreted as a non-linear continuity equation, includes the applications discussed in the introduction which exhibit saturation effects.
The term $\psi(\rho)$ is a \textit{saturation} of the density $\rho$ in the following sense:

\begin{definition}[Saturation]\label{def:saturation}
	A \textit{saturation} is a continuous function $\psi(s) : \Rplus\rightarrow \R$ with the properties: $\psi$ is non-increasing ($s_1 \leq s_2$ implies $\psi(s_1) \geq \psi(s_2)$); and there exists $\alpha>0$, called the \textit{saturation level}, such that $\psi(\alpha) = 0$ and $(\alpha-s)\psi(s)>0$ for $s\neq\alpha$.
\end{definition}

Typical examples of saturations are $\psi(s)=\alpha - s$ and $\psi(s)=(\alpha - s)|\alpha - s|^{m-1}$, $m>1$, commonly used in population models in mathematical biology \cite{J.B2002,C.M.S+2019}, as well as the social sciences \cite{B.P2016,B.B.R+2017}.
The case without saturation, $\psi(\rho)=1$, given by
\begin{equation}\label{eq:kineticgradientflow}
	\begin{cases}
		\pt\rho = \div\brk{\rho \grad \prt{H'(\rho) + V + W\conv\rho}}, \\
		\rho(0,\bx) = \rho_0(\bx),
	\end{cases}
\end{equation}
for $t>0$ and $\bx\in\Omega\subseteq \mathbb{R}^d$, possesses an interesting structural property: it is the \textit{gradient flow} of the energy functional
\begin{equation}\label{eq:kineticgradientflowenergy}
	E[\rho]=\int_{\Omega} \brk{H(\rho) + V\rho + \frac{1}{2}\prt{W\conv\rho}\rho}\dbx,
\end{equation}
see \cite{C.M.V2003,C.M.V2006}. In the 2-Wasserstein sense, this means that \cref{eq:kineticgradientflow} can be written as the non-linear continuity equation
\begin{equation}
	\pt \rho + \div\prt*{\rho \bu} = 0,\quad
	\bu =-\grad \xi,\quad
	\xi = \vder{E}{\rho},
\end{equation}
where $\vder{E}{\rho}$ is the first variation of the energy functional. \revisionOne{The variation is given by
	\begin{equation}
		\left.\der{}{\varepsilon} E\brk*{\rho + \varepsilon h}\right|_{\varepsilon=0} = \int_\Omega
		\vder{E}{\rho} h
		\dbx,
	\end{equation}
	for any $h$ such that $\int_\Omega h\dbx = 0$.} This structure leads to the dissipation of the energy along solutions of the equation,
\begin{equation}\label{eq:energyDissipation}
	\der{E}{t} = -\int_{\Omega} \rho\abs{\nabla \xi}^2\,\dd{\bx}\leq 0,
\end{equation}
and thus $E$ is a Lyapunov functional for the problem. Due to the relevance of \cref{eq:kineticgradientflow} in applications, numerical solutions which preserve this dissipation property are desirable; indeed \cite{B.C.H2020} introduces schemes to that effect.

\Cref{eq:kineticgradientflowsaturation} can no longer be obtained from the usual Wasserstein setting, but can nevertheless be formally cast as a gradient flow by writing
\begin{equation}
	\pt \rho + \div\prt*{\rho \psi\prt{\rho} \bu} = 0,\quad
	\bu =-\grad \xi,\quad
	\xi = \vder{E}{\rho},
\end{equation}
yielding the dissipation rate
\begin{equation}\label{eq:energyDissipationsaturation}
	\der{E}{t} = -\int_{\Omega} \rho\psi\prt{\rho}\abs{\nabla \xi}^2\,\dd{\bx}\leq 0
\end{equation}
for the energy \eqref{eq:kineticgradientflowenergy}. Under certain conditions on the saturation $\psi(\rho)$, \Cref{eq:kineticgradientflowsaturation} can be understood as a gradient flow with a suitably modified distance between probability measures \cite{CLSS10}.
As in the previous case, numerical schemes that preserve this dissipation structure are desirable; these will be introduced in \cref{sec:newScalarScheme}.

\revisionTwo{
	An interesting property of the gradient-flow structure of these equations is that the energy dissipation rate characterises their steady states.
	In the Wasserstein case \eqref{eq:kineticgradientflow}, the stationary solutions $\rho_\infty$ correspond to $\nabla \xi_\infty \equiv 0$ on their support.
	Likewise, in the case with saturation \eqref{eq:kineticgradientflowsaturation}, the steady states $\rho_\infty$ verify $\nabla \xi_\infty \equiv 0$ on the support of $\rho_\infty\psi\prt{\rho_\infty}$.
}

\subsection{Bounds on the Solution}\label{sec:newScalarBounds}

An important aspect in the study of \cref{eq:kineticgradientflowsaturation} are the bounds on the solution $\rho$. If the datum satisfies the bounds $0 \leq \rho_0(\bx) \leq \alpha$, we might hope that the solution would satisfy them too, provided sufficient assumptions on $\psi$, $H$, $V$, and $W$. This, in fact, is crucial to establishing any gradient structure, as the dissipation rate \eqref{eq:energyDissipationsaturation} hinges on the non-negativity of $\rho\psi\prt{\rho}$ (just as, in the case without saturation, the dissipation \eqref{eq:energyDissipation} relies on the non-negativity of $\rho$). Unfortunately, proving such bounds is no trivial matter, and it often requires a \revisionOne{case-by-case} analysis. Moreover, when we consider the extension of \cref{eq:kineticgradientflowsaturation} to systems in \cref{sec:newGradientFlowsSystems}, we will find that the existing literature is very sparse.

While a rigorous analysis is beyond the scope of this work, we will invoke the comparison principle for conservation laws of \cite{H.R.T1995,H.R2015} to justify the bounds in certain settings. For \cref{eq:kineticgradientflowsaturation}, the bounds $0 \leq \rho(t,\bx) \leq \alpha$ are satisfied if all of the following conditions hold:
\begin{enumerate}
	\item \cref{eq:kineticgradientflowsaturation} is posed on $\Rd$, or on a domain $\Omega$ with periodic boundary conditions;
	\item the datum satisfies the bounds $0 \leq \rho_0(\bx) \leq \alpha$;
	\item $\rho \psi\prt{\rho} \bu$ is continuous at $\rho = 0$ and $\rho = \alpha$;
	\item \revisionOne{$\div \prt{\rho \psi\prt{\rho} \bu}$} exists.
\end{enumerate}

The third and fourth conditions depend on the choice of $\psi$, $H$, $V$, and $W$. The common examples of saturations discussed above, $\psi(s)=\alpha - s$ and $\psi(s)=(\alpha - s)|\alpha - s|^{m-1}$, are both compatible. In applications, the diffusion term is typically $H(\rho)=\rho\prt{\log\rho-1}$ (the Boltzmann entropy, \revisionOne{often associated with the heat equation}) or $H(\rho)=\rho^m/(m-1)$ for some $m>1$ (\revisionOne{associated with the porous-medium equation}). These forms present no difficulty at $\rho=\alpha$, but require attention at $\rho=0$. In the \revisionOne{$H(\rho)=\rho\prt{\log\rho-1}$} case, it is preferable to rewrite the equation as $\pt\rho=\px\prt*{\psi(\rho)\px \rho}$ and to apply parabolic theory in order to obtain lower bounds. In the \revisionOne{$H(\rho)=\rho^m/(m-1)$} case, the question of regularity at the origin is quite involved, see \cite{Vazquez2006} for an exhaustive survey. This complexity permeates also the analysis of numerical schemes, as explored in \cite{B.C.M+2020}, which establishes the convergence of the scheme of \cite{B.C.H2020}. Concerning the potentials, it is sufficient to assume $V, W\in C^2(\Omega)$. This is stringent and could be relaxed in some cases; for instance, if the form of $H$ leads to establishing some regularity on the solution, it might be possible to interpret $\px\prt{-\px \prt{W\conv\rho}}$ as $-\px W\conv \px \rho$, requiring milder assumptions on $W$.

\subsection{Numerical Schemes}\label{sec:newScalarScheme}

If the bounds have been established on the solutions of \cref{eq:kineticgradientflowsaturation}, we can improve the numerical schemes for gradient flows from \cite{B.C.H2020} to include the problem with saturation by constructing a suitable discretisation of $\psi$. The crucial question we address here is the choice of the upwinding of the flux which preserves the saturation bounds. Two schemes were given in \cite{B.C.H2020}: schemes S1, and S2. In the interest of brevity, we describe the extension of S2 only, though the S1 can be similarly improved.

We will prescribe a full discretisation in the finite-volume setting in one dimension. The spatial domain $\prt{-L, L}$ is divided into $2M$ uniform volumes of size $\Dx = L/M$. We shall define the points $x\i = \Dx (i-1/2) - L$, $i=1,\cdots,2M$ to construct the $i$\th cell $C_i=\prt{x\imh, x\ih}$ with centre $x\i$. The time domain $\prt{0, T}$ is discretised through equispaced points with separation $\Dt=T/N$, with the $n$\th point given by $t\n=n\Dt$, $n=0,\cdots,N$.

The solution $\rho(t,x)$ is to be approximated by a function defined at each time step and constant over each cell; we denote by $\rho\i\n$ its value at time $t\n$ over the cell $C_i$. The new scheme reads:
\begin{subequations}\label{eq:S2saturation}
\begin{align+}
& \frac{\rho\i\np - \rho\i\n}{\Dt} + \frac{F\ih\np - F\imh\np}{\Dx} = 0, \\
& F\ih\np = \rho\i\np \pos{\psi\ip\np} \pos{u\ih\np} + \rho\ip\np \pos{\psi\i\np} \neg{u\ih\np}, \\
& u\ih\np=-\frac{\xi\ip\np-\xi\i\np}{\Dx},
\quad
\xi\i\np=H'(\rho\i\np)+V\i+(W\ast\rho\nss)\i,
\end{align+}
\end{subequations}
where $\pos{s}\coloneqq \max\set{0,s}$ and $\neg{s}\coloneqq \min\set{0,s}$. The saturation terms are given by $\psi\i\n = \psi\prt{\rho\i\n}$. The confining potential terms are defined as $V\i=V(x\i)$ or $V\i=\revisionOne{\frac{1}{\Dx}}\int_{C\i}V(s)\ds$. The convolution is given by $(W\ast\rho\nss)\i=\sum_{k}\Wik\rho_k\nss\Dx$, where $\Wik=W(x\i-x\k)$ or $\Wik=\revisionOne{\frac{1}{\Dx}}\int_{C\k}W(x\i-s)\ds$, and $\rho\nss=\prt{\rho\np+\rho\n}/{2}$; in practice, the sum will be evaluated through a fast Fourier transform. This scheme is first-order accurate in both time and space.

\revisionOne{We will often consider no-flux boundary conditions for the scheme, given by} $F_{\nhalf}\np = 0$ and $F_{2M+\nhalf}\np = 0$. These are beyond the scope of the theory of \cref{sec:newScalarBounds}, and therefore the bounds on the solution to the continuous problem are not justified; nevertheless, if the no-flux problem approximates the problem on the whole space well, a similar behaviour can be expected.

We begin the analysis of the scheme by proving it preserves the $0\leq\rho\leq\alpha$ bounds unconditionally:
\begin{proposition}[Boundedness and non-negativity]\label{th:implicitpositivity}
	For a given $n$, suppose $0\leq\rho\i\n\leq\alpha$ for all $i$. Then scheme \eqref{eq:S2saturation} satisfies $0\leq\rho\i\np\leq\alpha$ bounds unconditionally for all $i$.
\end{proposition}
\begin{proof}
	We will prove non-negativity followed by boundedness. Seeking a contradiction, we will assume $\rho\i\np < 0$ or $\rho\i\np > \alpha$ for some values of $i$. Without loss of generality we may assume that these ``pathological'' values lie on contiguous cells; if there are two or more separate pathological clusters, the proof we present below may be applied several times to deal with each one individually. Thus we may simply assume the values in question are $\rho\j\np, \rho\jp\np, \cdots, \rho\k\np$. Summing scheme \eqref{eq:S2saturation} over the corresponding cells yields
	\begin{equation}\label{eq:schemecontradiction}
		\sum_{i=j}^{k}\prt{\rho\i\np-\rho\i\n}\frac{\Dx}{\Dt}
		= - F\kh\np + F\jmh\np.
	\end{equation}

	To prove non-negativity, we assume $\rho\i\np$ is strictly negative for $j\leq i\leq k$. Since $\rho\i\n\geq 0$, the left hand side of \cref{eq:schemecontradiction} is strictly negative also. The right hand side is comprised of
	\revisionOne{
		\begin{align}
			- F\kh\np+F\jmh\np & =
			- {\rho\k\np} \pos{\psi\kp\np} u\kh^+ - {\rho\kp\np} \pos{\psi\k\np} u\kh^-
			\\&\qquad
			+
			{\rho\jm\np} \pos{\psi\j\np} u\jmh^+ + {\rho\j\np} \pos{\psi\jm\np} u\jmh^-.
		\end{align}
	}
	The first and fourth terms, $- {\rho\k\np} \pos{\psi\kp\np} u\kh^+$ and ${\rho\j\np} \pos{\psi\jm\np} u\jmh^-$, are non-negative, since $\rho\k\np,\rho\j\np<0$; the
	remaining terms are also non-negative. This follows because, outside the range $j\leq i\leq k$, the densities $\rho\i\np$ are non-negative (whether the upper bound is respected here is of no consequence). Therefore, the right hand side of \cref{eq:schemecontradiction} is also shown to be non-negative, producing a contradiction.

	Boundedness is proven in a similar fashion. We assume $\rho\i\np$ is strictly larger than $\alpha$ for $j\leq i\leq k$, resulting on $\psi\i\np\leq 0$, and rendering the left hand side of \cref{eq:schemecontradiction} strictly positive. On the right hand side, the second and third terms, $- {\rho\kp\np} \pos{\psi\k\np} u\kh^-$ and ${\rho\jm\np} \pos{\psi\j\np} u\jmh^+$, are identically zero because the saturation is non-positive. Furthermore, the first and fourth terms, $- {\rho\k\np} \pos{\psi\kp\np} u\kh^+$ and ${\rho\j\np} \pos{\psi\jm\np} u\jmh^-$, are non-positive. Thus, the right hand side is also non-positive, yielding a contradiction.
\end{proof}

\begin{remark}[Strict bounds]\label{th:strictBounds}
	It is possible to show a strict inequality version of \cref{th:implicitpositivity}: for a given $n$, suppose $0<\rho\i\n<\alpha$ for all $i$; then scheme \eqref{eq:S2saturation} satisfies $0<\rho\i\np<\alpha$ unconditionally for all $i$. This can be accomplished through an M-matrix argument, as was done for the positivity results in \cite{B.C.H2020}; for positivity, one proves $A\rho\np = \rho\n$ where $A$ is an inverse-positive matrix, whereas for boundedness one shows $B [\alpha\bone-\rho\np] = \alpha\bone-\rho\n$, where $B$ is also inverse positive and $\bone=\prt{1,\cdots,1}^\top$. However, for the latter, one needs the additional assumptions on the saturation; namely, that the quantity $\prt{\alpha - \rho}/{\alpha \psi\prt{\rho}}$ is bounded away from zero in the $\rho\goes\alpha$ limit. This is satisfied by the typical saturations mentioned above, $\psi(s)=\alpha - s$ and $\psi(s)=(\alpha - s)|\alpha - s|^{m-1}$ \revisionOne{for $m>1$}.
\end{remark}

Once the proof of bounds has been accomplished, we are ready to show the most important property of the scheme: that it preserves the energy dissipation structure of \Cref{eq:kineticgradientflowsaturation} unconditionally. We will define the discrete counterpart of the energy \eqref{eq:kineticgradientflowenergy} as
\begin{equation}\label{eq:kineticgradientflowenergydiscrete}
	E_\Delta[\rho\n] =
	\sum_{i=1}^{2M} H\prt{\rho\i\n} \Dx
	+\sum_{i=1}^{2M} V\i\rho\i\n \Dx
	+\frac{1}{2}\sum_{i=1}^{2M}\sum_{k=1}^{2M}\Wik\rho\i\n\rho\k\n \Dx^2,
\end{equation}
and show, as a straightforward corollary to \cite[Theorem 3.9]{B.C.H2020}, the dissipation result:
\begin{theorem}[Energy dissipation]\label{th:energydissipationS2}
	Scheme \eqref{eq:S2saturation} satisfies a fully discrete version of the energy dissipation property \eqref{eq:energyDissipation} unconditionally; namely, the discrete energy \eqref{eq:kineticgradientflowenergydiscrete} satisfies
	\begin{equation}
		\frac{E_\Delta[\rho\np] - E_\Delta[\rho\n]}{\Dt}
		\leq
		-\sum_{i=1}^{2M-1} \min\set*{\rho\i\np \pos{\psi\ip\np}, \rho\ip\np \pos{\psi\i\np}}|u\ih\np|^2 \Dx \leq 0.
	\end{equation}
\end{theorem}
\begin{proof}
	The proof of \cite[Theorem 3.9]{B.C.H2020} remains valid until the last two lines, where the inequality becomes
	\begin{align}
		     & \frac{E_\Delta[\rho\np]-E_\Delta [\rho\n]}{\Dt\Dx}                                                                     \\
		\leq & -\sum_{i=1}^{2M-1} \prt*{\rho\i\np \pos{\psi\ip\np} \pos{u\ih\np} + \rho\ip\np \pos{\psi\i\np} \neg{u\ih\np}} u\ih\np,
		\\ \leq & -\sum_{i=1}^{2M-1} \min\set*{\rho\i\np \pos{\psi\ip\np}, \rho\ip\np \pos{\psi\i\np}}|u\ih\np|^2
		\leq 0,
	\end{align}
	by virtue of \cref{th:implicitpositivity}.
\end{proof}

\revisionTwo{
	\begin{remark}
		\Cref{th:energydissipationS2} also characterises the steady states of the numerical scheme. By looking at the discrete dissipation rate, we see that the stationary condition is consistent to first-order with the continuous condition stated at the beginning of this section.
	\end{remark}
}

\begin{remark}[Second Order Schemes]
	It is also possible to develop structure-preserving schemes for \Cref{eq:kineticgradientflowsaturation} with second-order accuracy in space, as was done for \Cref{eq:kineticgradientflow} in \cite{CCH15} (a fully explicit scheme) and \cite{B.C.H2020} (a semi-implicit scheme). Both approaches will preserve the solution bounds, provided the CFL condition
	\begin{equation}\label{eq:explicitpositivity}
		\Dt \leq \Gamma \frac{\Dx}{2\max\i\abs{u\ih}},
		\quad
		\Gamma = \min\set*{\frac{1}{\psi(0)}, \gamma},
		\quad
		\gamma = \inf_{s\in\brk{0, \alpha}} \frac{\alpha - s}{\alpha\psi\prt{s}}
	\end{equation}
	is met. The explicit scheme will, however, not preserve the dissipative structure; dissipation can only be proven for a continuous-in-time, discrete-in-space scheme, as done in \cite{CCH15}. The semi-implicit scheme will preserve the energy dissipation property, but an implicit scheme with a CFL condition quickly becomes computationally impractical. The numerical example in \cref{sec:two-segments} demonstrates the advantages of the fully implicit scheme.
\end{remark}

\begin{remark}[Higher Dimensions]\label{th:higherdimensions}
	Scheme \eqref{eq:S2saturation} may be generalised directly to higher dimensions and, \textit{mutatis mutandis}, analogues of \cref{th:implicitpositivity,th:energydissipationS2} can be similarly proven. However, the computational cost of an implicit numerical scheme with non-local terms in multiple dimensions is impractically high. On the other hand, a direct dimensional splitting generalisation will capture the positivity and boundedness properties, but will not preserve the energy dissipation whenever the non-local terms are present.

	Yet, as was demonstrated in \cite{B.C.H2020}, there is a special form of dimensional splitting (\textit{sweeping dimensional splitting}) which generalises the one-dimensional scheme to the higher-dimensional setting and which retains the energy dissipation property even in the presence of the convolution term. That approach can be immediately adapted to our scheme, and we will employ it in all the examples of \cref{sec:numericalExperiments} to perform the computations efficiently.
\end{remark}
\section{Systems of Gradient Flows with Saturation}\label{sec:newGradientFlowsSystems}

We now turn our attention to the generalisation of \eqref{eq:kineticgradientflowsaturation}; the system of multiple species
\begin{equation}\label{eq:multispeciesgradientflow}
	\begin{cases}
		\pt \brho
		= \div\prt*{M\prt{\brho} \psi(\sigma) \grad\prt*{\vder{E}{\brho}}}
		= \sum_{l=1}^{d} \partial_{x\l}\prt*{M\prt{\brho} \psi(\sigma) \partial_{x\l}\prt*{\vder{E}{\brho}}},
		\\
		\brho(0,\bx) = \brho(\bx),
	\end{cases}
\end{equation}
for $t>0$ and $\bx\in\Omega\subseteq \mathbb{R}^d$.
This equation describes the evolution of the \textit{density} $\brho=\prt{\rho_1, \rho_2, \cdots, \rho_P}^\top$, a vector with entries $\rho_p\prt{t,\bx}: \Rplus\times\Omega \rightarrow \Rplus$ corresponding to the density of the $p$\th species. $M\prt{\brho}$ is a $P\times P$ positive semi-definite matrix (possibly non-symmetric). Additionally, $\psi(\s):\R^+\rightarrow\R$ is once again a \textit{saturation}, viz. Definition \ref{def:saturation}; in this case, the saturation depends on the \textit{total density} of the system, $\s\prt{t,\bx}\coloneqq\sum_{p=1}^{P} \rho_p\prt{t,\bx}$.

Once an energy functional $E[\brho]$ is prescribed, the system can be formally cast as a gradient flow by writing, as in the scalar case,
\begin{equation}
	\pt \brho + \div\prt*{M\prt{\brho} \psi\prt{\brho} \bu} = 0,\quad
	\bu =-\grad \bxi,\quad
	\bxi = \vder{E}{\brho},
\end{equation}
yielding the formal dissipation rate
\begin{equation}\label{eq:energyDissipationsaturationSystem}
	\der{}{t}E\brk{\brho}
	=
	\revisionOne{
		-\sum_{l=1}^{d}
		\int_\Omega \psi\prt{\sigma} \partial_{x\l}\prt*{\vder{E}{\brho}} \cdot M\prt{\brho} \partial_{x\l}\prt*{\vder{E}{\brho}} \dbx
	}
	\leq 0,
\end{equation}
since $M\prt{\brho}$ is positive semi-definite.

\revisionTwo{
	Just as in the scalar case, the dissipation rate characterises the steady states of \eqref{eq:multispeciesgradientflow}. The steady states $\brho_\infty$ verify $M(\brho_\infty)\partial_{x_l}\bxi_{\infty}\equiv 0$ for every $l$ on the support of $\psi(\sigma_\infty)$.
}

For brevity, we will consider here the case of two species, though the general case can be handled identically, and all the results shown in this section generalise immediately. Letting $\brho = \prt{\rho, \eta}^\top$, and $\sigma=\rho+\eta$, we define the energy functional as
\begin{align}\label{eq:systemenergy}
	E[\brho]       & = \curlyH[\brho] + \curlyV[\brho] + \curlyW[\brho],
	\\
	\curlyH[\brho] & = \int_\Omega \brk*{ H_\rho(\rho) + H_\eta(\eta) + H_\sigma(\sigma) }\dx,                                                            \\
	\curlyV[\brho] & = \int_\Omega \brk*{ V_\rho\rho + V_\eta\eta }\dx,                                                                                   \\
	\curlyW[\brho] & = \int_\Omega \brk*{ \frac{1}{2}\rho\prt{W_\rho\conv\rho} + \frac{1}{2}\eta\prt{W_\eta\conv\eta} + \rho\prt{W_\sigma\conv\eta} }\dx.
\end{align}
The internal energy densities $H_\rho(\rho)$, $H_\eta(\eta)$, and $H_\sigma(\sigma)$ are convex functions which model diffusion, respectively, on the first species, on the second, and on the total density of the system. The confining potentials $V_\rho$ and $V_\eta$ model external forces acting on each of the species. The symmetric interaction potentials $W_\rho$, $W_\eta$, and $W_\sigma$, model attraction or repulsion between particles; $W_\rho$ and $W_\eta$ represent the forces within each of the species, whereas $W_\sigma$ models the interaction between the species, assumed here to be symmetric.

\subsection{Bounds on the Solution}\label{sec:newSystemBounds}

As in the scalar case, a crucial aspect in the study of \cref{eq:multispeciesgradientflow} are the bounds on the solution. Here, for initial data such that $0 \leq \rho_0, \eta_0$ and $\sigma_0 \leq \alpha$, we might hope to find that the solution satisfies $0 \leq \rho, \eta$ and $\sigma \leq \alpha$. Once again, this is crucial to establishing any gradient structure.

We will make the necessary assumptions to apply the comparison principle of \cite{H.R.T1995} and establish the bounds. While the non-negativity of each species is found directly, the $\sigma \leq \alpha$ bound requires casting the system in a different form. \Cref{eq:multispeciesgradientflow} can be written as
\begin{equation}\label{eq:twoSystem}
	\begin{cases}
		\pt \rho + \div\prt{\rho\psi(\sigma) \svec{v}}=0, \\
		\pt \eta + \div\prt{\eta\psi(\sigma) \svec{w}}=0,
	\end{cases}
	\quad\text{where}\quad
	\begin{pmatrix}
		\svec{v} \\ \svec{w}
	\end{pmatrix} =
	\revisionOne{-}
	\diag\prt{\brho}^{-1} M\prt{\brho} \grad\prt*{\vder{E}{\brho}};
\end{equation}
it is important to ensure that these velocities are well-defined, either by making assumptions on $M$, by considering the case without cross-diffusion ($M\prt{\brho}=\diag\prt{\brho}$), or by establishing that the densities are bounded away from zero \textit{a priori}. The key to the upper bound is to replace $\sigma$ by $\alpha-\theta$, where $\alpha$ is the saturation level and $\theta$ is a new, independent variable. Letting $\tilde\psi(\theta) = \psi(\alpha-\theta)$, we define a new system
\begin{equation}\label{eq:twoSystemMod}
	\begin{cases}
		\pt \rho + \div\brk{\rho\tilde\psi(\theta) \svec{v}}=0, \\
		\pt \eta + \div\brk{\eta\tilde\psi(\theta) \svec{w}}=0, \\
		\pt \theta - \div\brk{\tilde\psi(\theta)
			\prt*{\rho \svec{v} + \eta \svec{w}}
		}=0.
	\end{cases}
\end{equation}
Crucially, any weak solution $(\rho,\eta,\theta)$ with datum $\rho_0+\eta_0+\theta_0\equiv\alpha$ satisfies $\rho+\eta+\theta\equiv\alpha$, establishing a one-to-one correspondence between solutions of \cref{eq:twoSystem} and solutions of \cref{eq:twoSystemMod} via $\sigma\equiv \alpha-\theta$.

The new system is in the correct form to apply the comparison principle of \cite{H.R.T1995}. We therefore recover, for any datum $\rho_0,\eta_0,\theta_0>0$, a viscosity solution $\rho,\eta,\theta>0$, understood here as the $\varepsilon\rightarrow 0$ limit of the solutions to the family of parabolic systems
\begin{equation}\label{eq:twoSystemModVisc}
	\begin{cases}
		\pt \rho + \div\brk{\rho\tilde\psi(\theta) \svec{v}} = \varepsilon \laplace\rho, \\
		\pt \eta + \div\brk{\eta\tilde\psi(\theta) \svec{w}} = \varepsilon \laplace\eta, \\
		\pt \theta - \div\brk{\tilde\psi(\theta)
			\prt*{\rho \svec{v} + \eta \svec{w}}
		} = \varepsilon \laplace\theta.
	\end{cases}
\end{equation}
In particular, solutions to \cref{eq:twoSystemModVisc} with datum $\rho_0+\eta_0+\theta_0\equiv\alpha$ correspond to solutions $\prt{\rho,\eta}$ to \cref{eq:twoSystem} which satisfy $0<\rho,\eta$ and $\rho+\eta<\alpha$. Moreover, these are also viscosity solutions of \eqref{eq:twoSystem}, which can be verified by establishing a similar correspondence between \cref{eq:twoSystemModVisc} and the parabolic version of \cref{eq:twoSystem},
\begin{equation}\label{eq:twoSystemVisc}
	\begin{cases}
		\pt \rho + \div\prt{\rho\psi(\sigma) \svec{v}}=\varepsilon \laplace\rho, \\
		\pt \eta + \div\prt{\eta\psi(\sigma) \svec{w}}=\varepsilon \laplace\eta.
	\end{cases}
\end{equation}
As a result, provided the assumptions stated in \cref{sec:newScalarBounds} are met, the solution to \cref{eq:multispeciesgradientflow} satisfies $0 \leq \rho, \eta$ and $\sigma \leq \alpha$.

\subsection{Numerical Schemes}\label{sec:newSystemScheme}

We can now extend scheme \eqref{eq:S2saturation} to the case of systems. The new scheme reads:
\begin{subequations}\label{eq:gradientSystemScheme}
\begin{align+}
& \frac{\brho\i\np-\brho\i\n}{\Dt} + \frac{\bF\ih\np-\bF\imh\np}{\Dx} = 0, \\
& \bF\ih\np = \diag\prt{{\brho\i\np}} \pos{\psi\ip\np} \pos{\bu\ih\np} + \diag\prt{{\brho\ip\np}} \pos{\psi\i\np} \neg{\bu\ih\np}, \\
& \bu\ih\np = D\ih\np \bv\ih\np, \quad \bv\ih\np = M\prt{\brho\ih\np} \bg\ih\np, \\
& \bg\ih\np = -\frac{ \bxi\ip\np - \bxi\i\np }{\Dx}
\end{align+}
where $\brho\ih\np = \prt{\brho\i\np + \brho\ip\np} / 2$. The terms $\pos{s}$ and $\neg{s}$ are as described in \cref{sec:newScalarScheme}, and defined entry-wise when applied to a vector.
The saturation terms are given by $\psi\i\n = \psi\prt{\sigma\i\n}$, where $\s\i\n=\sum_{p=1}^{P}\prt{\brho\i\n}_p$.
The diagonal matrix $D\ih\np$ is defined as
\begin{equation+}
D\ih\np =
\diag\prt{{\brho\i\np}}^{-1}
\He\prt{\bv\ih\np}
+
\diag\prt{{\brho\ip\np}}^{-1}
\He\prt{-\bv\ih\np},
\end{equation+}
where $\He$ is the Heaviside step function, applied to the vector $\bv\ih\np$ entry-wise.

In the two species case, the entropy variable $\xi_{\eta, i}\np$ is given by
\begin{align+}
& \bxi\i\np = \prt*{\xi_{\rho, i}\np, \xi_{\eta, i}\np}^\top, \\
& \xi_{\rho, i}\np =
H_\rho'\prt{\rho\i\np}
+ H_\sigma'\prt{\sigma\i\np}
+ V_{\rho,i}
+ (W_\rho\ast\rho\nss)\i
+ (W_\sigma\ast\eta\nss)\i, \\
& \xi_{\eta, i}\np =
H_\eta'\prt{\eta\i\np}
+ H_\sigma'\prt{\sigma\i\np}
+ V_{\eta,i}
+ (W_\eta\ast\eta\nss)\i
+ (W_\sigma\ast\rho\nss)\i,
\end{align+}
The terms of $V_{\rho,i}$, $V_{\eta,i}$, $(W_\rho\ast\rho\nss)\i$, $(W_\eta\ast\eta\nss)\i$, $(W_\sigma\ast\rho\nss)\i$, $(W_\sigma\ast\eta\nss)\i$, $\rho\nss$, and $\eta\nss$ are defined analogously to the terms of scheme \eqref{eq:S2saturation}, and the spatial discretisation is identical. These entropy terms can be immediately generalised to any number of species, and all the results presented in this section hold.
\end{subequations}

\begin{remark}
	In the absence of cross-diffusion effects, i.e. if $M\prt{\brho}=\diag\prt{\brho}$, the scheme can be simplified to bypass the matrix $D\ih\np$ altogether:
	\begin{align}
		 & \frac{\brho\i\np-\brho\i\n}{\Dt} + \frac{\bF\ih\np-\bF\imh\np}{\Dx} = 0,                                                         \\
		 & \bF\ih\np = \diag\prt{{\brho\i\np}} \pos{\psi\ip\np} \pos{\bu\ih\np} + \diag\prt{{\brho\ip\np}} \pos{\psi\i\np} \neg{\bu\ih\np}, \\
		 & \bu\ih\np = -\frac{ \bxi\ip\np - \bxi\i\np }{\Dx}.
	\end{align}
	Nevertheless, \cref{th:implicitpositivitysystem,th:energydissipationsystem} will be proven for the general version.
\end{remark}

\revisionTwo{
	\begin{remark}\label{th:vacuum}
		The definition of $D\ih\np$ can be altered slightly to handle the cases with vacuum:
		\begin{equation}
			D\ih\np =
			\diag\prt{\max\set{\brho\i\np,\varepsilon}}^{-1}
			\He\prt{\bv\ih\np}
			+
			\diag\prt{\max\set{\brho\ip\np,\varepsilon}}^{-1}
			\He\prt{-\bv\ih\np},
		\end{equation}
		where the maximum is applied entry-wise, and where $\varepsilon$ is a small parameter. In the simulations of \cref{sec:SKT}, we take $\varepsilon$ as the machine precision.
	\end{remark}
}

We begin the analysis by showing that this new scheme preserves the $0 \leq \brho$ and $\sigma \leq \alpha$ bounds unconditionally:
\begin{proposition}[Boundedness and non-negativity]\label{th:implicitpositivitysystem}
	For a given $n$, suppose $0\leq \brho\i\n$ (entry-wise) as well as $\s\i\n\leq\alpha$, for all $i$. Then scheme \eqref{eq:gradientSystemScheme} satisfies $0\leq \brho\i\np$ as well as $\s\i\np\leq\alpha$, unconditionally for all $i$.
\end{proposition}

\begin{proof}
	We will argue by contradiction, proving non-negativity followed by boundedness. As in the proof of \cref{th:implicitpositivity}, we will only deal with one cluster of contiguous pathological values, $j\leq i\leq k$. Summing scheme \eqref{eq:gradientSystemScheme} over the corresponding cells yields
	\begin{equation}\label{eq:schemesystemcontradiction}
		\sum_{i=j}^{k}\prt{\brho\i\np-\brho\i\n}\frac{\Dx}{\Dt}
		= - \bF\kh\np + \bF\jmh\np.
	\end{equation}

	We prove non-negativity one species at a time. Given $p$, suppose $\prt{\brho\i\np}_p$ is strictly negative for $j\leq i\leq k$, and non-negative otherwise. Whether the entries outside the pathological range satisfy the boundedness property is, for this part, irrelevant, as are the values of the other species. The $p$\th entry of \cref{eq:schemesystemcontradiction} is given by
	\begin{equation}\label{eq:schemesystemcontradictionentry}
		\sum_{i=j}^{k}\brk*{\prt{\brho\i\np}_p-\prt{\brho\i\n}_p}\frac{\Dx}{\Dt}
		= - \prt{\bF\kh\np}_p + \prt{\bF\jmh\np}_p.
	\end{equation}
	Since $\prt{\brho\i\n}_p\geq 0$, the left hand side of this equation is strictly negative. The right hand side is comprised of
	\begin{align}\label{eq:systemtermsone}
		- \prt{\bF\kh\np}_p & =
		- \prt{\brho\k\np}_p \pos{\psi\kp\np} \pos{\bu\kh}_p
		- \prt{\brho\kp\np}_p \pos{\psi\k\np} \neg{\bu\kh}_p,
		\\
		\label{eq:systemtermstwo}
		\prt{\bF\jmh\np}_p  & =
		\prt{\brho\jm\np}_p \pos{\psi\j\np} \pos{\bu\jmh}_p
		+ \prt{\brho\j\np}_p \pos{\psi\jm\np} \neg{\bu\jmh}_p.
	\end{align}
	Just as in the proof of \cref{th:implicitpositivity}, the first term of \eqref{eq:systemtermsone} and the second term of \eqref{eq:systemtermstwo} are non-negative, since $\prt{\brho\k\np}_p,\prt{\brho\j\np}_p<0$; the other terms are also guaranteed to be non-negative. Therefore, the right hand side of \cref{eq:schemesystemcontradictionentry} is shown to be non-negative, producing a contradiction.

	Boundedness is now proven in a similar fashion, by considering instead the sum of \cref{eq:schemesystemcontradiction} over all species:
	\begin{equation}\label{eq:schemesystemcontradictionsummed}
		\sum_{i=j}^{k}\prt{\s\i\np-\s\i\n}\frac{\Dx}{\Dt}
		= - \sum_{p=1}^{P} \prt{\bF\kh\np}_p + \sum_{p=1}^{P} \prt{\bF\jmh\np}_p.
	\end{equation}
	Here, we assume that $\s\i\np$ is strictly larger than $\alpha$ for $j\leq i\leq k$, resulting on $\psi\i\np\leq 0$, and rendering the left hand side of \eqref{eq:schemesystemcontradictionsummed} strictly positive. The right hand side now comprises four terms, corresponding to the sums over $p$ of \cref{eq:systemtermsone,eq:systemtermstwo}. Two of the terms, $- \sum_{p}\prt{\brho\kp\np}_p \pos{\psi\k\np} \neg{\bu\kh}_p$ and $\sum_{p}\prt{\brho\jm\np}_p \pos{\psi\j\np} \pos{\bu\jmh}_p$, are identically zero; the other terms
	are guaranteed to be non-positive. Thus, the right hand side is also non-positive, yielding a contradiction.
\end{proof}

\begin{remark}[Strict Bounds]
	As in \cref{th:strictBounds}, strict bounds on the solution can be found if the initial datum also satisfies them. This, in particular, ensures that the matrix $D\ih\np$ is well-defined for cross-diffusion systems without vacuum.
\end{remark}

Armed with the bounds, we arrive at the crux of the analysis: the unconditional dissipation structure of the scheme. We will define the discrete counterpart of the energy \eqref{eq:systemenergy} as
\begin{align}\label{eq:kineticgradientflowenergydiscreteSystem}
	E_\Delta[\brho\n]       & = \curlyH_\Delta[\brho\n] + \curlyV_\Delta[\brho\n] + \curlyW_\Delta[\brho\n],
	\\
	\curlyH_\Delta[\brho\n] & = \sum_{i=1}^{2M} \brk*{
		H_\rho\prt{\rho\i\n} + H_\eta\prt{\eta\i\n} + H_\sigma\prt{\sigma\i\n}
	} \Dx,                                                                                                   \\
	\curlyV_\Delta[\brho\n] & = \sum_{i=1}^{2M} \brk*{
		V_{\rho, i} \rho\i\n + V_{\eta, i} \eta\i\n
	} \Dx,                                                                                                   \\
	\curlyW_\Delta[\brho\n] & = \frac{1}{2} \sum_{i=1}^{2M} \sum_{k=1}^{2M} \brk*{
		\rho\i\n \rho\k\n W_{\rho, i-k} +
		\eta\i\n \eta\k\n W_{\eta, i-k} +
		2\rho\i\n \eta\k\n W_{\sigma, i-k}
	} \Dx^2,
\end{align}
and show:
\begin{theorem}[Energy dissipation]\label{th:energydissipationsystem}
	Scheme \eqref{eq:gradientSystemScheme} satisfies a fully discrete version of the energy dissipation property \eqref{eq:energyDissipationsaturationSystem} unconditionally; namely, the discrete energy \eqref{eq:kineticgradientflowenergydiscreteSystem} satisfies
	\begin{equation}
		\frac{E_\Delta[\rho\np] - E_\Delta[\rho\n]}{\Dt}
		\leq - \sum_{i=1}^{2M-1} \bg\ih\np \cdot M\prt{\brho\ih} \min\set{\psi\i\np,\psi\ip\np} \bg\ih\np \Dx
		\leq 0.
	\end{equation}
\end{theorem}

\begin{proof}
	We multiply the scheme by the entropy variable $\bxi\i\np$ and \revisionOne{sum over the spatial domain} to find
	\begin{equation}
		\sum_{i=1}^{2M} \bxi\i\np \cdot \prt{\brho\i\np-\brho\i\n} = -\frac{\Dt}{\Dx} \sum_{i=1}^{2M} \bxi\i\np \cdot \prt{\bF\ih\np-\bF\imh\np},
	\end{equation}
	which can be rewritten as
	\begin{align}\label{eq:VContributionIdentity}
		 & \quad \sum_{i=1}^{2M} \bV\i \cdot \prt{\brho\i\np-\brho\i\n}                  \\ \nonumber
		 & = -\frac{\Dt}{\Dx} \sum_{i=1}^{2M} \bxi\i\np \cdot \prt{\bF\ih\np-\bF\imh\np}
		- \sum_{i=1}^{2M} H_\rho'\prt{\rho\i\np} \prt{\rho\i\np-\rho\i\n}                \\ \nonumber
		 & \quad - \sum_{i=1}^{2M} H_\eta'\prt{\eta\i\np} \prt{\eta\i\np-\eta\i\n}
		- \sum_{i=1}^{2M} H_\sigma'\prt{\sigma\i\np} \prt{\sigma\i\np-\sigma\i\n}        \\ \nonumber
		 & \quad - \sum_{i=1}^{2M} \prt{\rho\i\np-\rho\i\n}(W_\rho\ast\rho\nss)\i
		- \sum_{i=1}^{2M} \prt{\rho\i\np-\rho\i\n}(W_\sigma\ast\eta\nss)\i               \\ \nonumber
		 & \quad - \sum_{i=1}^{2M} \prt{\eta\i\np-\eta\i\n}(W_\eta\ast\eta\nss)\i
		- \sum_{i=1}^{2M} \prt{\eta\i\np-\eta\i\n}(W_\sigma\ast\rho\nss)\i,
	\end{align}
	where $\bV\i=\prt{V_{\rho,i}, V_{\eta,i}}^\top$. On the other hand, the evolution of the discrete energy, $E_\Delta(\brho\np) - E_\Delta(\brho\n)$, is equal to the sum of $\curlyH_\Delta(\brho\np) - \curlyH_\Delta(\brho\n)$, $\curlyV_\Delta(\brho\np) - \curlyV_\Delta(\brho\n)$, and $\curlyW_\Delta(\brho\np) - \curlyW_\Delta(\brho\n)$. The contribution corresponding to $\curlyH$ is
	\begin{align}
		 & \quad \curlyH_\Delta(\brho\np) - \curlyH_\Delta(\brho\n) \\
		 & = \sum_{i=1}^{2M} \brk*{
			H_\rho\prt{\rho\i\np} - H_\rho\prt{\rho\i\n} + H_\eta\prt{\eta\i\np} - H_\eta\prt{\eta\i\n} + H_\sigma\prt{\sigma\i\np} - H_\sigma\prt{\sigma\i\n}
		} \Dx;
	\end{align}
	the contribution of $\curlyV$ is
	\begin{align}
		\curlyV_\Delta(\brho\np) - \curlyV_\Delta(\brho\n)
		 & = \sum_{i=1}^{2M} \brk*{
			V_{\rho, i} \prt{\rho\i\np - \rho\i\n} + V_{\eta, i} \prt{\eta\i\np - \eta\i\n}
		} \Dx
		\\ &
		= \sum_{i=1}^{2M} \bV\i \cdot \prt{\brho\i\np-\brho\i\n};
	\end{align}
	finally, the contribution of $\curlyW$ is
	\begin{align}
		\curlyW_\Delta(\brho\np) - \curlyW_\Delta(\brho\n)
		 & = \frac{1}{2} \sum_{i=1}^{2M} \sum_{k=1}^{2M} \prt*{ \rho\i\np \rho\k\np - \rho\i\n \rho\k\n } W_{\rho, i-k} \Dx^2      \\
		 & \quad +\frac{1}{2} \sum_{i=1}^{2M} \sum_{k=1}^{2M} \prt*{ \eta\i\np \eta\k\np - \eta\i\n \eta\k\n } W_{\eta, i-k} \Dx^2 \\
		 & \quad + \sum_{i=1}^{2M} \sum_{k=1}^{2M} \prt*{ \rho\i\np \eta\k\np - \rho\i\n \eta\k\n } W_{\sigma, i-k} \Dx^2.
	\end{align}
	The second contribution, that pertaining to $\curlyV$, corresponds to the identity \cref{eq:VContributionIdentity}. Through its substitution, we may rewrite the evolution of the discrete energy as the sum of three terms:
	$
		E_\Delta[\brho\np] - E_\Delta[\brho\n] = I + II + III.
	$

	The first energy term involves only entropies:
	\begin{align}
		I
		 & = \sum_{i=1}^{2M} \brk*{H_\rho\prt{\rho\i\np} - H_\rho\prt{\rho\i\n} - H_\rho'\prt{\rho\i\np} \prt{\rho\i\np-\rho\i\n}} \Dx                        \\
		 & \quad + \sum_{i=1}^{2M} \brk*{H_\eta\prt{\eta\i\np} - H_\eta\prt{\eta\i\n} - H_\eta'\prt{\eta\i\np} \prt{\eta\i\np-\eta\i\n}} \Dx                  \\
		 & \quad + \sum_{i=1}^{2M} \brk*{H_\sigma\prt{\sigma\i\np} - H_\sigma\prt{\sigma\i\n} - H_\sigma'\prt{\sigma\i\np} \prt{\sigma\i\np-\sigma\i\n}} \Dx.
	\end{align}
	This quantity is immediately controlled, using the convexity of $H_\rho$, $H_\eta$, and $H_\sigma$; for instance, $H_\rho\prt{\rho\i\np} - H_\rho\prt{\rho\i\n} - H_\rho'\prt{\rho\i\np} \prt{\rho\i\np-\rho\i\n} \leq 0$. Applying this to every term, we find $I \leq 0$.

	The second energy term involves only potentials:
	\begin{align}
		II
		 & = \frac{1}{2} \sum_{i=1}^{2M} \sum_{k=1}^{2M} \prt*{ \rho\i\np \rho\k\np - \rho\i\n \rho\k\n - \prt{\rho\i\np-\rho\i\n}\prt{\rho\k\np+\rho\k\n} } W_{\rho, i-k} \Dx^2         \\
		 & \quad +\frac{1}{2} \sum_{i=1}^{2M} \sum_{k=1}^{2M} \prt*{ \eta\i\np \eta\k\np - \eta\i\n \eta\k\n - \prt{\eta\i\np-\eta\i\n}\prt{\eta\k\np+\eta\k\n} } W_{\eta, i-k} \Dx^2    \\
		 & \quad +\frac{1}{2} \sum_{i=1}^{2M} \sum_{k=1}^{2M} \prt*{ \rho\i\np \eta\k\np - \rho\i\n \eta\k\n - \prt{\rho\i\np-\rho\i\n}\prt{\eta\k\np+\eta\k\n} } W_{\sigma, i-k} \Dx^2  \\
		 & \quad +\frac{1}{2} \sum_{i=1}^{2M} \sum_{k=1}^{2M} \prt*{ \rho\i\np \eta\k\np - \rho\i\n \eta\k\n - \prt{\eta\i\np-\eta\i\n}\prt{\rho\k\np+\rho\k\n} } W_{\sigma, i-k} \Dx^2.
	\end{align}
	Exploiting the symmetric character of the potentials, e.g. $W_{\sigma, i-k} = W_{\sigma, k-i}$, this term is shown to be identically zero.

	The last energy term involves only fluxes:
	\begin{align}
		III & = -\Dt \sum_{i=1}^{2M} \bxi\i\np \cdot \prt{\bF\ih\np-\bF\imh\np}
		= \Dt \sum_{i=1}^{2M-1} \prt{\bxi\ip\np - \bxi\i\np} \cdot \bF\ih\np,
	\end{align}
	performing summation by parts and invoking the zero-flux boundary conditions. Applying the definition of the numerical flux, each summand is
	\begin{equation}
		\prt{\bxi\ip\np - \bxi\i\np} \cdot \prt*{\diag\prt{{\brho\i\np}} {\psi\ip\np} \pos{\bu\ih\np} + \diag\prt{{\brho\ip\np}} {\psi\i\np} \neg{\bu\ih\np}}.
	\end{equation}
	Note that we write ${\psi\i\np}$ and ${\psi\ip\np}$, rather than $\pos{\psi\i\np}$ and $\pos{\psi\ip\np}$, because the bounds of \cref{th:implicitpositivitysystem} imply that these terms cannot be negative. The summand equals
	\begin{equation}
		\prt{\bxi\ip\np - \bxi\i\np} \cdot \prt{{\psi\ip\np} \pos{\bv\ih\np} + {\psi\i\np} \neg{\bv\ih\np}},
	\end{equation}
	using the definition of $D\ih\np$, and the fact that the $p\th$ entry of the vectors $\bu\ih\np$ and $\bv\ih\np$ have the same sign by construction.
	Employing, in turn, the definitions of $\bv\ih\np$ and $\bg\ih\np$, and the positive-semi-definiteness of $M$, this term is bounded above by
	\begin{equation}
		- \bg\ih\np \cdot M\prt{\brho\ih\np} \min\set{\psi\i\np,\psi\ip\np} \bg\ih\np \Dx \leq 0,
	\end{equation}
	therefore controlling the sum:
	\begin{equation}
		III
		\leq
		-\Dt \sum_{i=1}^{2M-1} \bg\ih\np \cdot M\prt{\brho\ih\np} \min\set{\psi\i\np,\psi\ip\np} \bg\ih\np \Dx
		\leq 0.
	\end{equation}
	This finally yields
	\begin{equation}
		E_\Delta[\brho\np] - E_\Delta[\brho\n] = I + II + III \leq III \leq 0.
	\end{equation}

\end{proof}

\begin{remark}[Higher Dimensions]
	Scheme \eqref{eq:gradientSystemScheme} will be extended to the higher-dimensional setting through the process described in \cref{th:higherdimensions}.
\end{remark}

\section{Numerical Experiments}
\label{sec:numericalExperiments}

This section is devoted to showcasing the properties of our numerical schemes in several challenging problems. We will use the implicit scheme \eqref{eq:gradientSystemScheme} for all examples, except for \cref{sec:two-segments}, where we compare the performance of both the implicit scheme \eqref{eq:S2saturation} and its explicit version. We shall prioritise numerical experiments involving systems of equations, as these are both more challenging and more interesting.

We first verify the numerical accuracy of our implicit scheme on systems with and without saturation in \cref{sec:SKT}. \Cref{sec:two-segments} describes the effect of saturation in a gradient flow with an external potential, a case not directly covered by the theory of gradient flows with non-linear mobility \cite{CLSS10}. We numerically explore the qualitative properties of the solutions, showing that saturation leads to the convergence towards stationary states with kinks at the free boundary of the fully saturated region. \Cref{sec:freeze} is devoted to another novel qualitative effect for systems due to saturation, what we call the \textit{freeze-in-place} phenomenon. Because of the upper bound on the density, certain populations, whose fate was to segregate in the absence of saturation effects, do not achieve complete segregation, as the total population reaches the saturation level, leading to a ``frozen'' configuration. Finally, \cref{sec:mastercard} showcases cell-cell adhesion sorting phenomena, based on the different strengths of attraction between two subpopulations. The different attractions between the species lead to segregation, partial engulfment, or complete engulfment, as was reported in \cite{C.M.S+2019} for compactly supported attractive kernels.

\subsection{Cross-Diffusion Experiment}
\label{sec:SKT}

\placedfigure{
\includegraphics{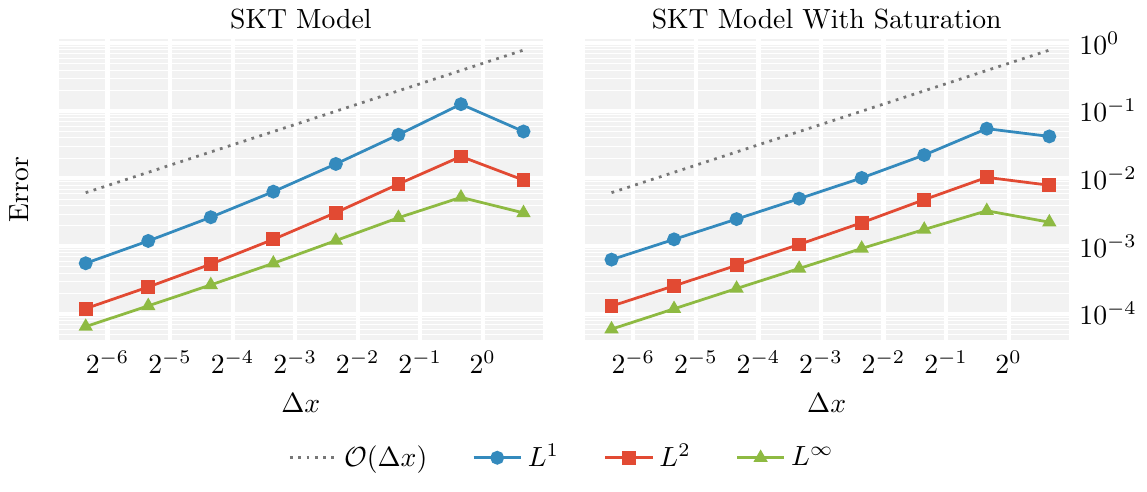}
\caption{
SKT model.
$\Lone$, $\Ltwo$, and $\Linf$ errors of the numerical solution of scheme \eqref{eq:gradientSystemScheme} with respect to the manufactured solution \eqref{eq:manufactured}.
\textbf{Left}: regular model \eqref{eq:SKTGradientFlow}.
\textbf{Right}: model with saturation \eqref{eq:SKTGradientFlowSaturation}.
$\Dx=2^{-k}\pi$ and $\Dt=2^{-k}/10$ for $k=1,\cdots,8$.
}
\label{fig:SKT}
}

The first experiment aims to verify the numerical accuracy of the scheme \eqref{eq:gradientSystemScheme} on a problem with cross diffusion, a feature not present in the scalar tests found in \cite{B.C.H2020}. To that end, we shall employ the SKT population model \cite{S.K.T1979}, given by
\begin{equation}\label{eq:SKT}
	\begin{cases}
		\pt\rho = \div\brk{\prt*{2\rho+\eta}\grad\rho + \rho\grad\eta \vphantom{{x^2}^2} } + s_\rho, \\
		\pt\eta = \div\brk{\eta\grad\rho + \prt*{\rho+2\eta}\grad\eta \vphantom{{x^2}^2} } + s_\eta,
	\end{cases}
\end{equation}
where $\bs=\prt{s_{\rho}, s_{\eta}}^\top$ is a given source term. We will consider the problem on $\Omega = (-\pi,\pi)^2$, posed with periodic boundary conditions.
\Cref{eq:SKT} can be cast in an analogous form to \eqref{eq:multispeciesgradientflow} with right-hand sides, written as
\begin{equation}\label{eq:SKTGradientFlow}
	\pt
	\begin{pmatrix} \rho \\ \eta \end{pmatrix}
	= \div \brk*{
		M\prt{\rho,\eta}
		\grad \begin{pmatrix} \log\rho \\ \log\eta \end{pmatrix}
	} +
	\begin{pmatrix} s_{\rho} \\ s_{\eta} \end{pmatrix}
	,
\end{equation}
where
\begin{equation}
	M\prt{\rho,\eta}
	=
	\begin{pmatrix}
		\rho\prt{2\rho + \eta} & \rho\eta               \\
		\rho\eta               & \prt{\rho + 2\eta}\eta
	\end{pmatrix}
	=
	\begin{pmatrix}
		\rho & 0    \\
		0    & \eta
	\end{pmatrix}
	\begin{pmatrix}
		2\rho + \eta & \eta         \\
		\rho         & \rho + 2\eta
	\end{pmatrix}.
\end{equation}
If $\bs=0$, then this is a gradient flow with respect to the energy
\begin{equation}
	E = \int_{\Omega} \rho\prt*{\log\rho-1} + \eta\prt*{\log\eta-1} \dbx.
\end{equation}

In order to perform the validation, we will use manufactured solutions. We assume the solution is of the form
\begin{equation}\label{eq:manufactured}
	\rho\prt{t,x,y} = \frac{1}{4}\brk*{1 + \sin(x + t)},\quad
	\eta\prt{t,x,y} = \frac{1}{4}\brk*{1 + \cos(y + t)},
\end{equation}
and compute the corresponding source terms,
\begin{align}
	s_{\rho} & =
	\frac{1}{4} \cos\left( t + x \right)
	+ \frac{3}{16} \sin\left( t + x \right)
	+ \frac{1}{16} \cos\left( t + y \right)
	\\&\quad
	- \frac{1}{8} \cos^{2}\left( t + x \right)
	+ \frac{1}{8} \sin^{2}\left( t + x \right)
	+ \frac{1}{8} \sin\left( t + x \right) \cos\left( t + y \right),
\end{align}
and
\begin{align}
	s_{\eta} & =
	\frac{1}{16} \sin\left( t + x \right)
	+ \frac{3}{16} \cos\left( t + y \right)
	- \frac{1}{4} \sin\left( t + y \right)
	\\&\quad
	+ \frac{1}{8} \cos^{2}\left( t + y \right)
	- \frac{1}{8} \sin^{2}\left( t + y \right)
	+ \frac{1}{8} \sin\left( t + x \right) \cos\left( t + y \right).
\end{align}
We will also perform the validation on a version of \Cref{eq:SKTGradientFlow} with a saturation term:
\begin{equation}\label{eq:SKTGradientFlowSaturation}
	\pt
	\begin{pmatrix} \rho \\ \eta \end{pmatrix}
	= \div \brk*{
		\prt*{1-\rho-\eta}
		M\prt{\rho,\eta}
		\grad \begin{pmatrix} \log\rho \\ \log\eta \end{pmatrix}
	} +
	\begin{pmatrix} s_{\rho} \\ s_{\eta} \end{pmatrix}
	.
\end{equation}
Assuming the same solution, \eqref{eq:manufactured}, yields the sources
\begin{align}
	s_{\rho} & =
	\frac{1}{4} \cos\left( t + x \right)
	+ \frac{3}{32} \sin\left( t + x \right)
	+ \frac{1}{32} \cos\left( t + y \right)
	- \frac{1}{64} \cos^{2}\left( t + x \right)
	+ \frac{1}{64} \sin^{2}\left( t + x \right)
	\\&\quad
	- \frac{1}{32} \sin^{3}\left( t + x \right)
	- \frac{1}{64} \cos^{2}\left( t + y \right)
	+ \frac{1}{64} \sin^{2}\left( t + y \right)
	+ \frac{1}{16} \cos^{2}\left( t + x \right) \sin\left( t + x \right)
	\\&\quad
	- \frac{1}{32} \cos^{2}\left( t + y \right) \sin\left( t + x \right)
	+ \frac{1}{64} \sin^{2}\left( t + y \right) \sin\left( t + x \right)
	\\&\quad
	+ \frac{3}{64} \cos^{2}\left( t + x \right) \cos\left( t + y \right)
	- \frac{1}{16} \sin^{2}\left( t + x \right) \cos\left( t + y \right),
\end{align}
and
\begin{align}
	s_{\eta} & =
	\frac{1}{32} \sin\left( t + x \right)
	+ \frac{3}{32} \cos\left( t + y \right)
	- \frac{1}{4} \sin\left( t + y \right)
	+ \frac{1}{64} \cos^{2}\left( t + x \right)
	- \frac{1}{64} \sin^{2}\left( t + x \right)
	\\&\quad
	+ \frac{1}{64} \cos^{2}\left( t + y \right)
	- \frac{1}{32} \cos^{3}\left( t + y \right)
	- \frac{1}{64} \sin^{2}\left( t + y \right)
	- \frac{1}{16} \cos^{2}\left( t + y \right) \sin\left( t + x \right)
	\\&\quad
	+ \frac{3}{64} \sin^{2}\left( t + y \right) \sin\left( t + x \right)
	+ \frac{1}{64} \cos^{2}\left( t + x \right) \cos\left( t + y \right)
	\\&\quad
	- \frac{1}{32} \sin^{2}\left( t + x \right) \cos\left( t + y \right)
	+ \frac{1}{16} \sin^{2}\left( t + y \right) \cos\left( t + y \right).
\end{align}

We solve \cref{eq:SKTGradientFlow,eq:SKTGradientFlowSaturation} over the interval $t\in\prt{0,0.1}$, choosing $\Dx=\Dy=2^{-k}\pi$ and $\Dt=2^{-k}/10$ for $k=1,\cdots,8$, and compute the $\Lp$ error ($p=1,2,\infty$) at the final time with respect to the solution \eqref{eq:manufactured}.
\revisionTwo{To handle the points where the solutions touch zero, we proceed as described in \cref{th:vacuum}}.
\Cref{fig:SKT} shows the results of the validation; scheme \eqref{eq:gradientSystemScheme} clearly approximates the manufactured solution \eqref{eq:manufactured} with first-order accuracy, \revisionTwo{despite the vacuum modification}.
\subsection{Saturation Experiment}
\label{sec:two-segments}

\placedfigure{
	\includegraphics{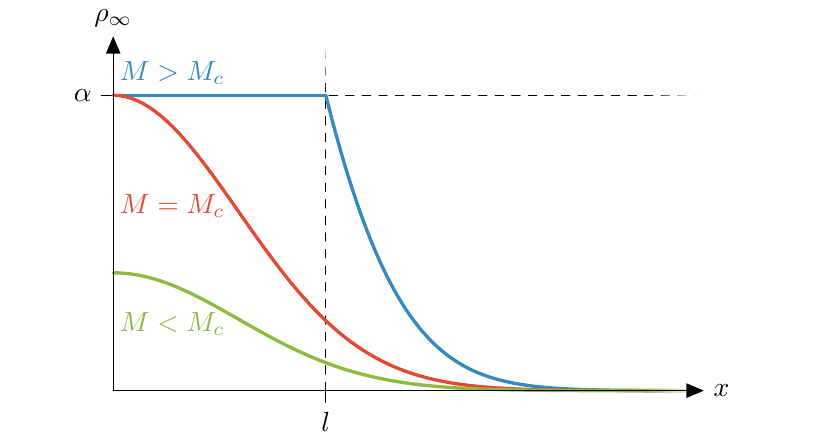}
	\caption{
		Stationary solutions to \cref{eq:segments} in one dimension.
		In the cases $M<M_c$ and $M=M_c$, the solution is a half-Gaussian. For $M>M_c$, the solution is made of a straight saturated segment and a Gaussian tail.
	}
	\label{fig:segmentDiagram}
}
\placedfigure{
	\includegraphics{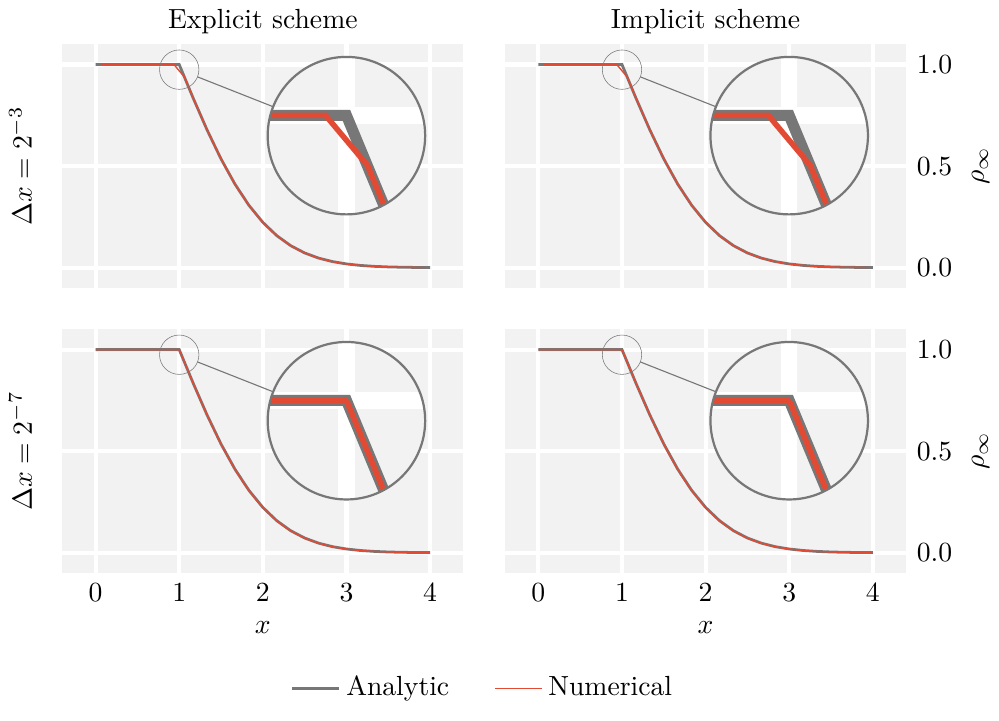}
	\caption{
		Stationary solution of \cref{eq:segments} in one dimension with mass $M=1.66$ corresponding to $l=1$.
		$\Omega=\prt{0,4}$,
		$\alpha=1$,
		$D=1$,
		$C=1$.
		Explicit: $\Dt=\Dx^2/4$.
		Implicit: $\Dt=\Dx$.
	}
	\label{fig:segments1D}
} \placedfigure{
\includegraphics{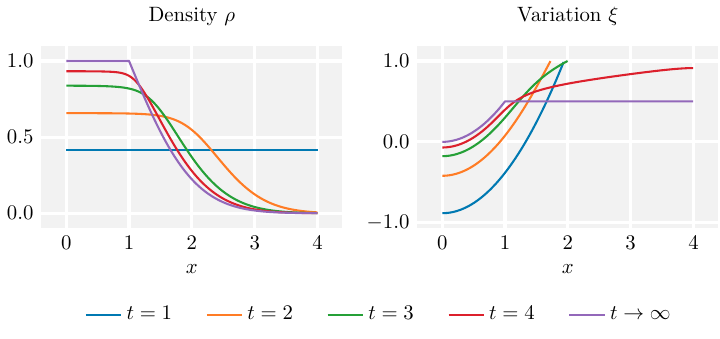}
\caption{
\revisionTwo{
Solution of \cref{eq:segments} in one dimension with mass $M=1.66$ corresponding to $l=1$.
$\Omega=\prt{0,4}$,
$\alpha=1$,
$D=1$,
$C=1$, $\Dt=\Dx=2^{-7}$.
}
}
\label{fig:segmentsNew}
} \placedfigure{
\includegraphics[width=0.75\linewidth]{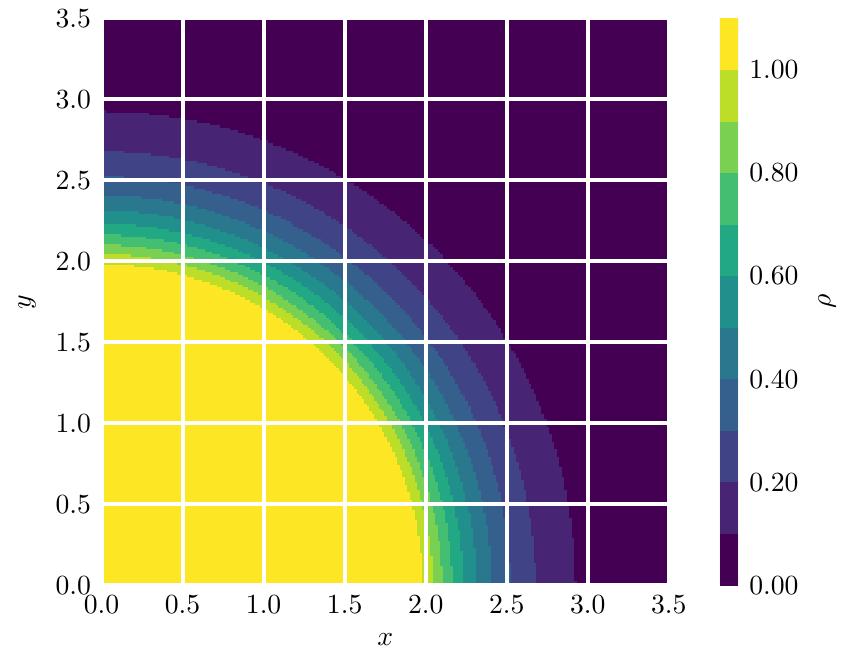}
\caption{
Stationary solution of \cref{eq:segments} in two dimensions with mass $M=4.71$ corresponding to $l=2$.
$\Omega=\prt{0,4}^2$,
$\alpha=1$,
$D=1$,
$C=1$,
$\Dt=\Dx=2^{-7}$.
}
\label{fig:segments2D}
}
The second experiment concerns the saturation effects. In order to demonstrate the boundedness property of the numerical schemes, we consider the equation
\begin{equation}\label{eq:segments}
	\pt\rho = \div\prt*{\rho\prt{\alpha-\rho} \grad\prt*{D\log\prt{\rho}+\frac{C}{2}\abs{\bx}^2}}.
\end{equation}
This problem exhibits different steady states which depend on $\revisionOne{m}\coloneqq \pnorm{1}{\rho_0}$, the (conserved) mass of the solution. If $\revisionOne{m}$ is small ($\revisionOne{m}<\revisionOne{m}_c$ for some critical mass $\revisionOne{m}_c$), the steady state is simply
$
	\rho_\infty = A\exp\prt*{-\frac{C}{2D}\abs{\bx}^2},
$
where $A$ is a positive constant such that $\pnorm{1}{\rho_\infty}=\revisionOne{m}=\pnorm{1}{\rho_0}$. However, if the mass is beyond the critical value, the steady state consists of two segments:
\begin{equation}
	\rho_\infty =
	\begin{cases}
		\alpha                                             & \text{if } \abs{\bx}\leq l, \\
		\displaystyle B\exp\prt*{-\frac{C}{2D}\abs{\bx}^2} & \text{if } \abs{\bx}\geq l,
	\end{cases}
\end{equation}
where $B$ is again a normalisation constant, and $l$ is to be determined, see \cref{fig:segmentDiagram} for a diagram. Observing that $\rho_\infty = \alpha$ whenever $\abs{\bx}=l$, we determine $B=\alpha\exp\prt*{\frac{Cl^2}{2D}}$; thus the steady state is rewritten as
\begin{equation}\label{eq:sswkink}
	\rho_\infty\prt{\bx} = \alpha \exp\prt*{-\frac{C}{2D}\pos{\abs{\bx}^2-l^2}},
\end{equation}
where $\pos{s}=\max\set{s,0}$ for any $s\in\R$. The value of $l$ is an increasing function of $\revisionOne{m}$, and can be determined form the initial datum. It is left to the reader to check \eqref{eq:sswkink} are indeed weak stationary solutions to \eqref{eq:segments}.

We may further explore the normalization condition.
We first consider \cref{eq:segments} in one dimension, posed on $\Omega = \prt{0,\infty}$; by virtue of the conservation of mass, we arrive at the self-consistency equation
\begin{align}
	\revisionOne{m} & = \int_{0}^{\infty} \alpha \exp\prt*{-\frac{C}{2D}\pos{x^2-l^2}} \dx = \alpha\brk*{ l + \int_{l}^{\infty} \exp\prt*{\frac{C}{2D}\prt{l^2-x^2}} \dx } \\
	                & = \alpha\brk*{ l + \sqrt{\frac{\pi D}{2C}}\exp\prt*{\frac{C}{2D}l^2}\brk*{1- \erf\prt*{l\sqrt{\frac{C}{2D}}}} }.
\end{align}
In particular, this expression yields the value of the critical mass,
$\revisionOne{m}_c = \alpha\sqrt{\frac{\pi D}{2C}}$,
which corresponds to $l=0$.
In the two dimensional case, we pose \cref{eq:segments} on the domain $\Omega = \prt{0,\infty}^2$. We find
\begin{align}
	\revisionOne{m}
	 & = \int_{0}^{\infty}\!\!\! \int_{0}^{\infty}
	\!\!\alpha \exp\prt*{-\frac{C}{2D}\pos{x^2+y^2-l^2}}
	\dx \dy                                        \\&= \int_{0}^{\frac{\pi}{2}} \!\!\!\int_{0}^{\infty}
	\!\!\alpha r\exp\prt*{-\frac{C}{2D}\pos{r^2-l^2}}
	\dr \dtheta                                    \\
	 & = \frac{\alpha\pi}{2} \brk*{
		\int_{0}^{l} r \dr
		+
		\int_{l}^{\infty} r\exp\prt*{-\frac{C}{2D}\prt{r^2-l^2}} \dr
	} = \frac{\alpha\pi}{2} \brk*{
		\frac{l^2}{2}
		+\frac{D}{C}
	},
\end{align}
yielding
$\revisionOne{m}_c = \frac{\alpha \pi D}{2C}.$

We will solve \cref{eq:segments} in one and two dimensions, over sufficiently large but finite domains, with supercritical masses, in order to validate the behaviour of the schemes in the presence of saturation. The problems are initialised with a constant density $\rho = \revisionOne{m}/\abs{\Omega}$, where $\abs{\Omega}$ is the Lebesgue measure of $\Omega$; the solution is computed for $t\in(0,15)$, and the final state is taken as an approximation of the steady state. We will employ Scheme \eqref{eq:gradientSystemScheme}, as well as an explicit counterpart found by evaluating the numerical fluxes at the $n\th$ time.

In one dimension, we let $\Omega=(0,4)$, \revisionOne{prescribe no-flux boundary conditions}, and choose the supercritical mass $\revisionOne{m}=1.66$, which corresponds to $l=1$. \Cref{fig:segments1D} shows the numerical steady state, computed with both the explicit and implicit schemes, each on a coarse ($\Dx=2^{-3}$) and a fine ($\Dx=2^{-7}$) mesh. Both schemes capture the upper bound of the solution, but the implicit scheme permits the use of a much larger time step.
\revisionTwo{\Cref{fig:segmentsNew} shows the evolution in time of the solution with the implicit scheme, as well as the evolution of the variation $\xi$. As expected, as the solution approaches the steady state, $\xi$ tends to a constant on the support of $\rho_\infty\psi(\rho_\infty)$.}

In two dimensions, we set $\Omega=(0,4)^2$, \revisionOne{again with no-flux boundary conditions}, and choose the supercritical mass $\revisionOne{m}=4.71$, found by letting $l=2$. \Cref{fig:segments2D} shows the numerical steady state computed with the implicit scheme.

We would like to emphasise that there is no general weak solution theory yet that establishes the kink solutions \eqref{eq:sswkink} as the global asymptotic profile for generic initial data to \eqref{eq:segments} with supercritical mass, though this behaviour was observed numerically for diverse initial data. The Cauchy theory for these general gradient flows \eqref{eq:segments} with non-linear mobility and external interaction potentials, not included in \cite{CLSS10}, is challenging and will be studied elsewhere. \subsection{Freeze-In-Place Experiment}
\label{sec:freeze}

\placedfigure{
\includegraphics[width=0.6\linewidth]{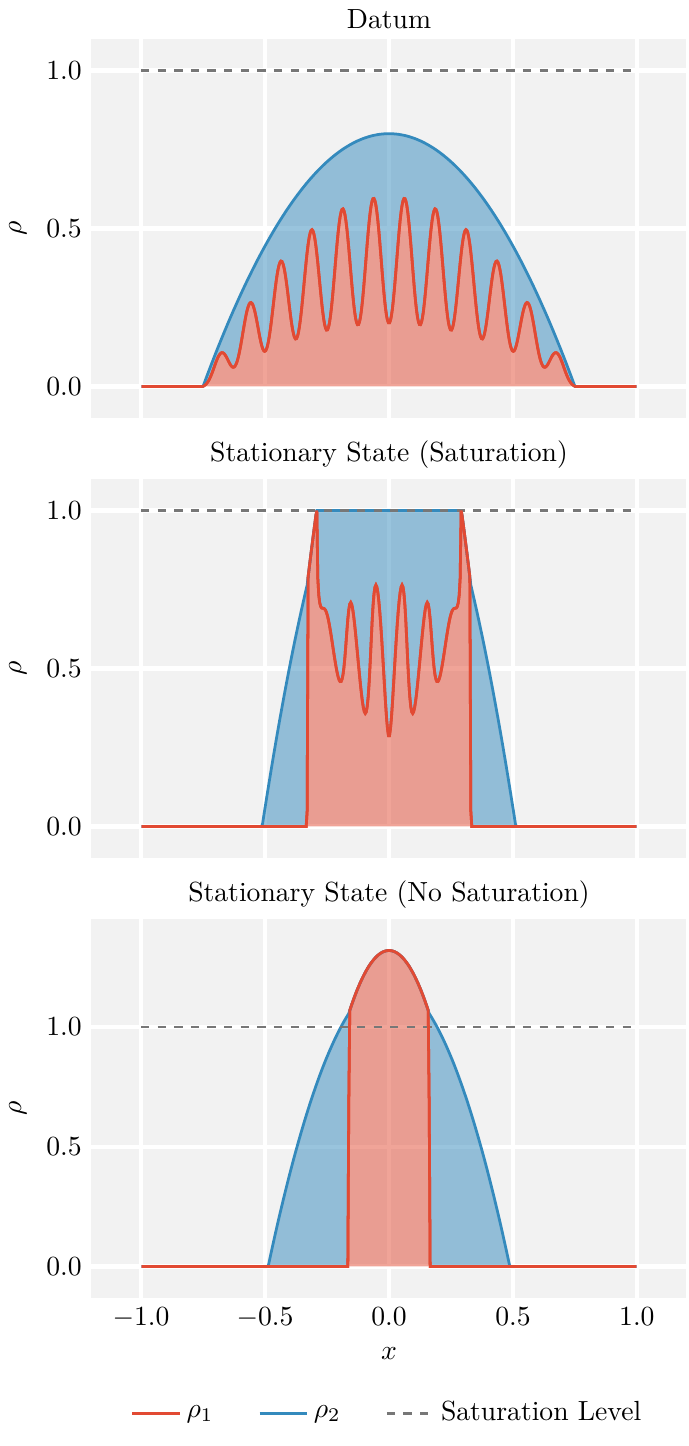}
\caption{
Stationary solutions of \cref{eq:squiggles} in one dimension.
$\Omega=\prt{-1,1}$,
$\alpha=1$,
$D=0.1$,
$C_1=4$,
$C_2=2$,
$\Dx=2^{-8}$,
and $\Dt=0.1$.
\textbf{Top}: datum \eqref{eq:squigglesDatum1D}.
\textbf{Middle}: ``frozen-in-place'' stationary state with saturation effects.
\textbf{Bottom}: segregated stationary state without saturation.
}
\label{fig:squiggles1D}
} \placedfigure{
\includegraphics[width=0.95\linewidth]{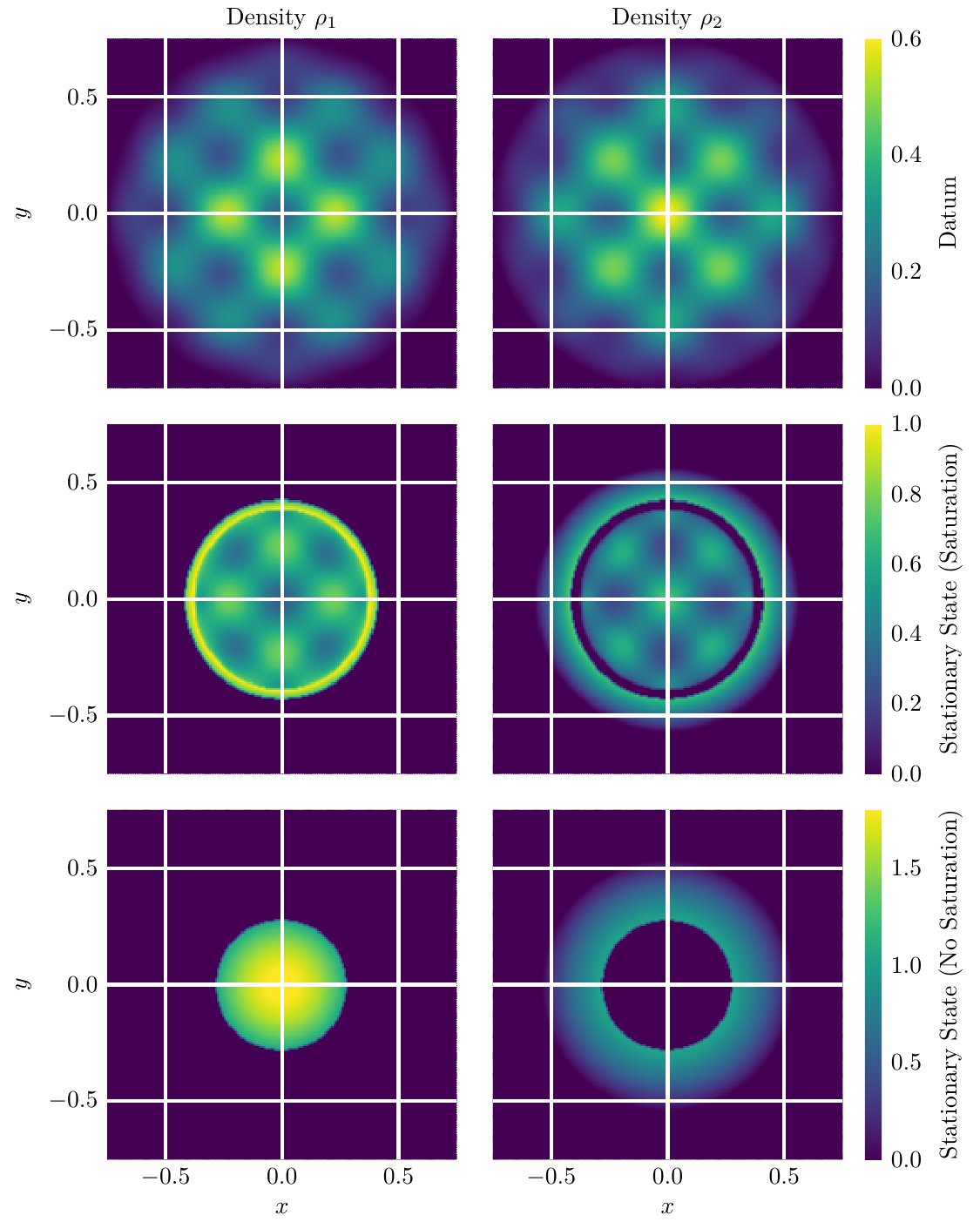}
\caption{
Stationary solutions of \cref{eq:squiggles} in two dimensions.
$\Omega=\prt{-1,1}^2$,
$\alpha=1$,
$D=0.1$,
$C_1=4$,
$C_2=2$,
$\Dx=2^{-7}$,
and $\Dt=0.1$.
\textbf{Top}: datum \eqref{eq:squigglesDatum1D}.
\textbf{Middle}: ``frozen-in-place'' stationary state with saturation effects.
\textbf{Bottom}: segregated stationary state without saturation.
}
\label{fig:squiggles2D}
}
This test will explore the appearance of infinitely many steady states as a result of the saturation effects. To that end, we will consider the two species system
\begin{equation}\label{eq:squiggles}
	\begin{cases}
		\pt\rho = \div\brk*{\rho\prt{\alpha-\sigma} \grad\prt*{D\sigma+\frac{C_1}{2}\abs{\bx}^2}}, \\
		\pt\eta = \div\brk*{\eta\prt{\alpha-\sigma} \grad\prt*{D\sigma+\frac{C_2}{2}\abs{\bx}^2}}.
	\end{cases}
\end{equation}
There are two effects present here: diffusion and confinement. The non-linear diffusion acts on the sum, $\sigma= \rho+\eta$. This can be seen by ignoring the $\abs{\bx}^2$ terms; then, the two equations sum to $\pt \sigma = D\div\prt*{\sigma\prt{\alpha-\sigma} \grad\sigma}$. Hence, the diffusion terms act to smooth irregularities on $\sigma$, but not necessarily on the individual profiles. Meanwhile, the quadratic potentials serve to confine each species, driving their densities towards the origin.

Crucially, we will choose $C_1 = 2C_2$; the confinement effects will therefore act more strongly on $\rho$ than on $\eta$. This difference in strengths, combined with the diffusion, will cause $\rho$ to displace $\eta$ near the origin. Below the saturation level, the species are expected to segregate completely; however, this might not occur if the species sum reaches the saturation level. As $\sigma$ approaches $\alpha$, the mobility reduces drastically, which will cause the solution to ``freeze'' in whatever configuration it may find itself.

We will let $D=0.1$, $C_1=4$, $C_2=2$, and $\alpha=1$. In one dimension, we solve the problem on $\Omega=(-1,1)$, with the initial datum
\begin{equation}\label{eq:squigglesDatum1D}
	\begin{cases}
		\rho_1(0,x) = \brk*{f(x)\prt*{1-\frac{\cos(\omega x)}{2}}}^+, \\
		\rho_2(0,x) = \brk*{f(x)\prt*{1+\frac{\cos(\omega x)}{2}}}^+,
	\end{cases}
	\qquad
	\text{where }
	f(x) = \frac{4}{5}\prt*{1-\prt*{\frac{4x}{3}}^2},
\end{equation}
and $\omega = 16\pi$. We will solve the system for $t\in\prt{0,30}$ and take the final state as an approximation to the asymptotic steady state which corresponds to this datum. \Cref{fig:squiggles1D} shows the datum, the ``frozen-in-place'' steady state, as well as the segregated steady state of the analogous system without saturation, all computed with the implicit scheme \eqref{eq:gradientSystemScheme}.

In two dimensions, we solve the problem on $\Omega=(-1,1)^2$. The datum now reads
\begin{equation}\label{eq:squigglesDatum2D}
	\begin{cases}
		\rho_1(0,x,y) = \brk*{f\prt*{\sqrt{x^2+y^2}}\prt*{1-\frac{\cos(\omega x)\cos(\omega y)}{2}}}^+, \\
		\rho_2(0,x,y) = \brk*{f\prt*{\sqrt{x^2+y^2}}\prt*{1+\frac{\cos(\omega x)\cos(\omega y)}{2}}}^+,
	\end{cases}
\end{equation}
where $\omega = 4\pi$ and $f$ is defined as above. Once again, we take the solution at $t=30$ as an approximation to the asymptotic steady state. \Cref{fig:squiggles2D} shows the datum as well as the steady states with and without saturation. As in the one-dimensional case, the saturation induces a ``freezing-in-place'' effect, preventing the complete segregation of the species. \subsection{Cell-Cell Adhesion Experiment}
\label{sec:mastercard}

\placedfigure{
\includegraphics{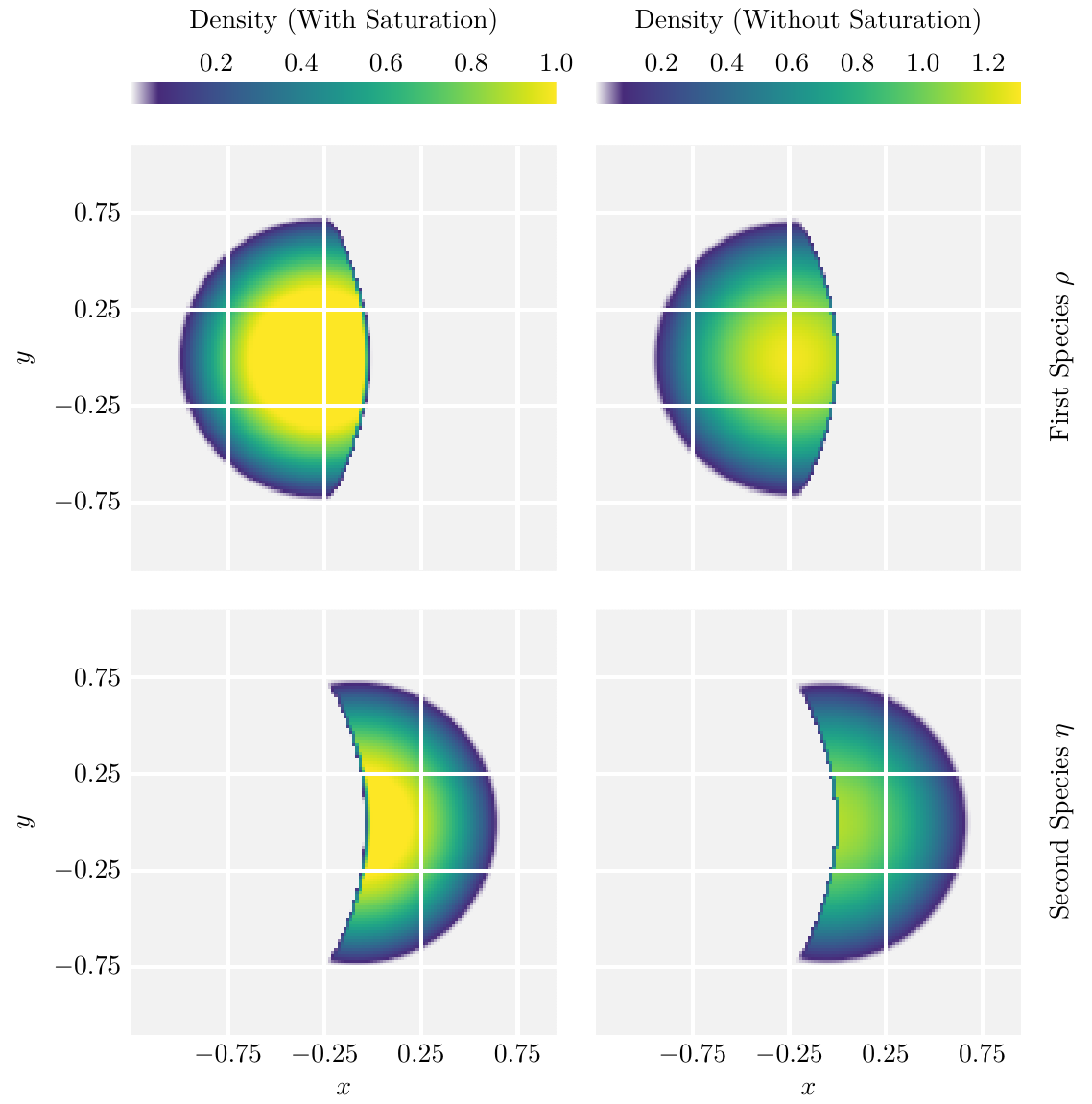}
\caption{
Steady state of \cref{eq:mastercard,eq:mastercardSaturation} in two dimensions.
$\Omega=\prt{-2,2}^2$,
$\alpha=1$,
$\varepsilon=0.1$,
$c_{\rho\rho}=2c_{\rho\eta}=2c_{\eta\rho}=c_{\eta\eta}=1$,
$\Dx=\Dy=\Dt=2^{-6}$.
\textbf{Left}: \cref{eq:mastercardSaturation} (saturation).
\textbf{Right}: \cref{eq:mastercard} (no saturation).
}
\label{fig:mastercardPartial}
} \placedfigure{
\includegraphics{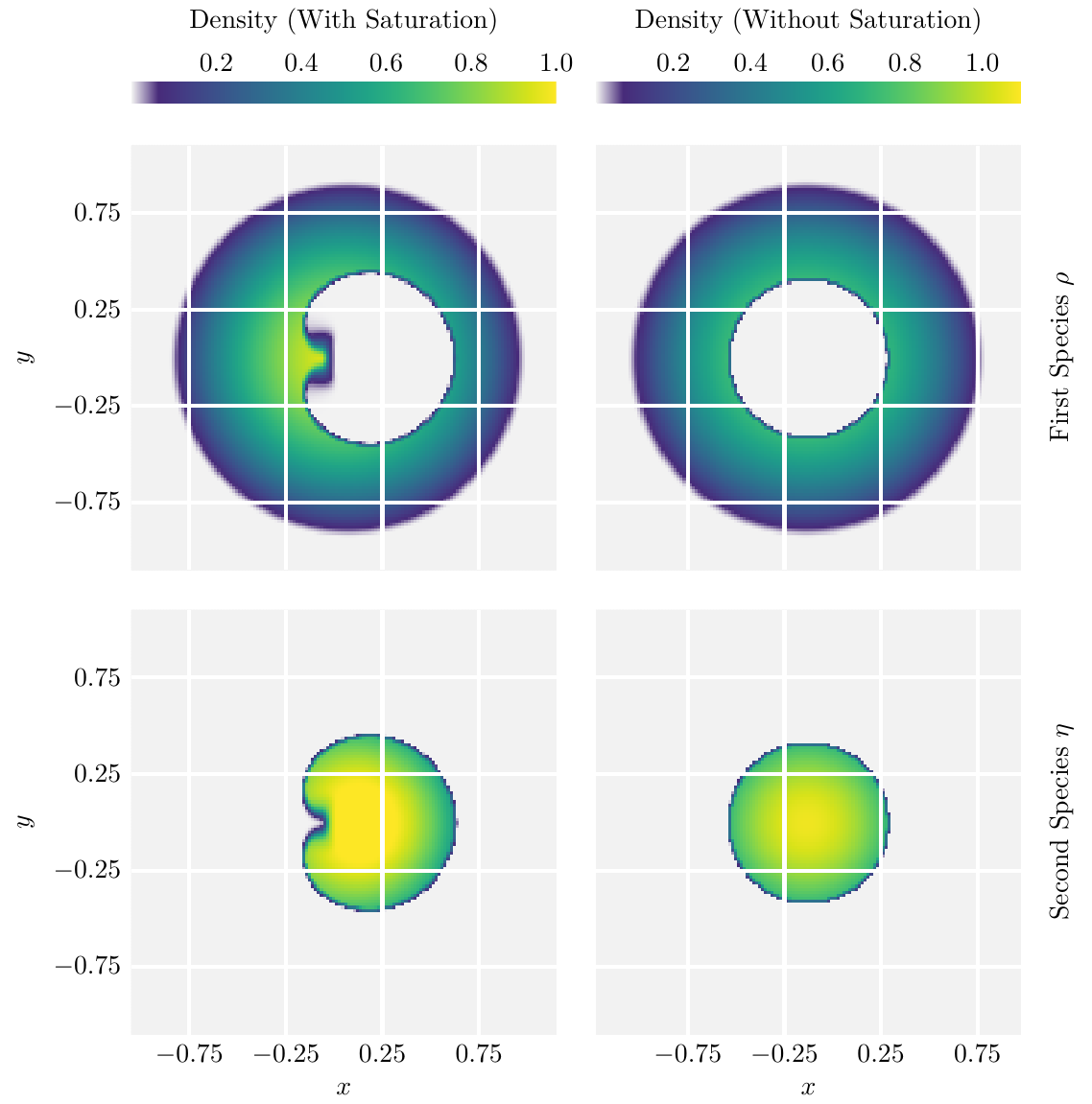}
\caption{
Steady state of \cref{eq:mastercard,eq:mastercardSaturation} in two dimensions.
$\Omega=\prt{-2,2}^2$,
$\alpha=1$,
$\varepsilon=0.1$,
$4c_{\rho\rho}=2c_{\rho\eta}=2c_{\eta\rho}=c_{\eta\eta}=1$,
$\Dx=\Dy=\Dt=2^{-6}$.
\textbf{Left}: \cref{eq:mastercardSaturation} (saturation).
\textbf{Right}: \cref{eq:mastercard} (no saturation).
}
\label{fig:mastercardEngulfment}
} \placedfigure{
	\includegraphics{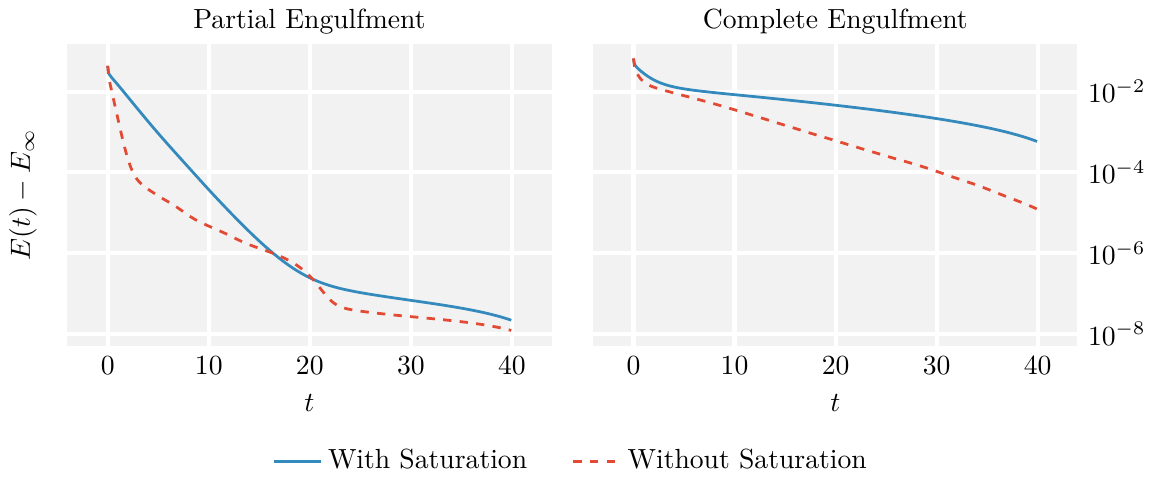}
	\caption{
		Dissipation of the free energy in the solutions to \cref{eq:mastercard,eq:mastercardSaturation} shown in \cref{fig:mastercardPartial,fig:mastercardEngulfment}.
		\textbf{Left}: partial engulfment, corresponding to \cref{fig:mastercardPartial}.
		\textbf{Right}: complete engulfment, corresponding to \cref{fig:mastercardEngulfment}.
	}
	\label{fig:mastercardEnergy}
}
To conclude the experiments section, we explore a version of the two-species model for cell-cell adhesion introduced in \cite{C.M.S+2019}. The model for the evolution of the densities of two different species, $\rho$ and $\eta$, is given by
\begin{equation}\label{eq:mastercard}
	\begin{cases}
		\pt\rho = \div\brk*{\rho\grad\prt{
				\varepsilon\sigma
				+ c_{\rho\rho} K \conv \rho
				+ c_{\rho\eta} K \conv \eta
			}}, \\
		\pt\eta = \div\brk*{\eta\grad\prt{
				\varepsilon\sigma
				+ c_{\eta\rho} K \conv \rho
				+ c_{\eta\eta} K \conv \eta
			}},
	\end{cases}
\end{equation}
in two spatial dimensions, for a positive value of $\varepsilon$, non-negative values of $c_{\rho\rho}$, $c_{\rho\eta}$, $c_{\eta\rho}$, and $c_{\eta\eta}$, and where $K(\bx) = \abs{\bx}^2/2$. The parameter $\varepsilon$ regulates the strength of non-linear diffusion on the total density of the system, which in this model accounts for localised repulsion effects. Each of the coefficients $c_{\alpha\beta}$ correspond to the strength of attraction of the species $\alpha$ to species $\beta$.

This model exhibits different asymptotic behaviours, regulated by the values of $c_{\alpha\beta}$. As discussed in \cite{C.M.S+2019}, different values of the parameters may lead to the mixing of the two species, a complete segregation, or the engulfment of one of the species (either partial or complete). We will concentrate here on the latter case, and will observe how the inclusion of saturation effects affect these behaviours, through the modified model
\begin{equation}\label{eq:mastercardSaturation}
	\begin{cases}
		\pt\rho = \div\brk*{\rho\prt{\alpha-\sigma}\grad\prt{
				\varepsilon\sigma
				+ c_{\rho\rho} K \conv \rho
				+ c_{\rho\eta} K \conv \eta
			}}, \\
		\pt\eta = \div\brk*{\eta\prt{\alpha-\sigma}\grad\prt{
				\varepsilon\sigma
				+ c_{\eta\rho} K \conv \rho
				+ c_{\eta\eta} K \conv \eta
			}}.
	\end{cases}
\end{equation}

We will prescribe an initial datum consisting of two touching disks:
\begin{equation}\label{eq:mastercardDatum2D}
	\begin{cases}
		\rho_0 = 0.95\,\mathds{1}_{\mathcal{D}_\rho}, \\
		\eta_0 = 0.95\,\mathds{1}_{\mathcal{D}_\eta},
	\end{cases}
\end{equation}
where $\mathcal{D}_\rho$ is a disk centred at $\prt{-0.5,0}$ with radius $0.5$, \revisionOne{$\mathcal{D}_\eta$} is a disk centred at \revisionOne{$\prt{0.4,0}$} with radius $0.4$, and $\mathds{1}$ indicates the characteristic function. We will let $\varepsilon=0.1$, and $\alpha=1$ when applicable. We shall solve \cref{eq:mastercard,eq:mastercardSaturation} on $\Omega=\prt{-2,2}^2$, over the interval $t\in\prt{0,45}$, and take their final state as an approximation of the asymptotic steady state. We choose $\Dx=\Dy=\Dt=2^{-6}$ to ensure high accuracy.

We first explore the partial engulfment case, given by $c_{\rho\rho}=2c_{\rho\eta}=2c_{\eta\rho}=c_{\eta\eta}=1$, where the two species are expected to partially surround each other in a lunar shape. \Cref{fig:mastercardPartial} shows the numerical steady states. In the system with saturation, the steady state exhibits a saturated, flat region. Nevertheless, both solutions are relatively similar.

We turn to the complete engulfment case, given by $4c_{\rho\rho}=2c_{\rho\eta}=2c_{\eta\rho}=c_{\eta\eta}=1$. Here, one species is expected to completely surround the other; in the long-time limit, one species will be supported on a disk, and the other, on an annulus. \Cref{fig:mastercardEngulfment} shows the numerical steady states. This time, the effects of the saturation are very visible: the initial concentration of the density creates a region of saturation where the two species remain mixed for all time, in the same vein as the example from \cref{sec:freeze}. This prevents the full segregation seen in the unsaturated case, and yields a cardioid-like free boundary between the species.

We conclude by showing the evolution of the (relative) free energy corresponding to the both the partial engulfment and the complete engulfment scenarios in \cref{fig:mastercardEnergy}. The energy is clearly dissipated by the numerical scheme, and the dissipation rates are affected by the presence of the saturation terms. In the partial engulfment case, the energies with and without saturation evolve very differently, but eventually reach a comparable value; this is consistent with the similar (though not equal) steady states reached in either setting. Meanwhile, the energies of the complete engulfment case demonstrate similar evolutions, but the dissipation rate is higher in the unsaturated setting.
\section*{Acknowledgements}
RB was funded by Labex CEMPI (ANR-11-LABX-0007-01). RB and JAC were supported by the Advanced Grant Nonlocal-CPD (Nonlocal PDEs for Complex Particle Dynamics: Phase Transitions, Patterns and Synchronization) of the European Research Council Executive Agency (ERC) under the European Union's Horizon 2020 research and innovation programme (grant agreement No. 883363). JH was partially supported under the NSF CAREER grant DMS-2153208 and AFOSR grant FA9550-21-1-0358.

\FloatBarrier

{
	\small
	\bibliographystyle{abbrv}
	\bibliography{./BailoCarrilloHu_BoundPreservingSchemes.bib}
}

\end{document}